\documentclass{article}
\usepackage{amsmath,enumerate,amsfonts,amssymb,color,graphicx,amsthm}
\renewcommand{\raggedright}{\leftskip=0pt \rightskip=0pt plus 0cm}

\usepackage{multirow,caption,subcaption}
\usepackage{slashed}
\usepackage{mathrsfs}
\usepackage{textcomp}
\usepackage{extarrows}
\usepackage{authblk}
\usepackage{titlesec}
\allowdisplaybreaks[4]
\usepackage{fancyhdr}
\usepackage{verbatim}
\usepackage{hyperref}
\usepackage{framed}
\usepackage[normalem]{ulem}
\usepackage{tikz}
\usetikzlibrary{patterns}
\usepackage[margin=1.25in]{geometry}
\usepackage{cite}
\usepackage{hyperref}
\hypersetup{
    colorlinks,
    citecolor=blue,
    filecolor=blue,
    linkcolor=blue,
    urlcolor=black
}
\numberwithin{equation}{section}
\newtheorem{thm}{Theorem}[section]
\newtheorem{lem}{Lemma}[section]
\newtheorem{prop}{Proposition}[section]
\newtheorem{defn}{Definition}[section]
\newtheorem{cor}{Corollary}[section]
\newtheorem{rem}{Remark}[section]

\newcommand{\R}{\mathbb{R}}

\newcommand{\Rn}{\mathbb{R}^n}

\newcommand{\be}{\begin{equation}}
\newcommand{\ee}{\end{equation}}

\numberwithin{equation}{section}

\begin{document}
\title{\bf Blow-up analysis and compactness for critical Dirac equations on $\mathbb S^n$
}

\author{2}
\author{\medskip  {\sc Ting Li},\ \ {\sc Zhongwei Tang}\footnote{Z. Tang is supported by National Science Foundation of China (12471099, 12271223).}  }
\date{}
\maketitle\begin{abstract}
We study blow-up and compactness for the critical nonlinear
Dirac equation
\[
D_{g_0}\psi_i
=
H_i|\psi_i|^{\frac{2}{n-1}}\psi_i
\qquad \text{on }\mathbb S^n,\quad n\ge 3,
\]
where $g_0$ is the round metric and
$H_i\to H>0$ in $C^0(\mathbb S^n)$.
Under a first-threshold bound on the critical mass,
we prove that every blowing-up sequence has, after passing
to a subsequence, a unique blow-up point $p$ satisfying
$H(p)=\max_{\mathbb S^n}H$.
The rescaled solutions converge to a ground-state Dirac
bubble that carries the entire limiting critical mass.
Under local uniform $C^1$ bounds on $H_i$, we use the Kelvin transform and a second rescaling to prove
that $p$ is an isolated simple blow-up point.
We also derive sharp pointwise and gradient estimates
in cylindrical coordinates and identify the normalized
Green limit.

If the coefficients converge in $C^2$ near $p$,
the Pohozaev identity yields
$\operatorname{Hess}_{g_0}H(p)=0$.
Consequently, if a positive $H\in C^2(\mathbb S^n)$
satisfies $\Delta_{g_0}H(p)<0$ at every global maximum,
then solutions with critical mass bounded by a level
slightly above the first threshold are uniformly bounded
in $L^\infty(\mathbb S^n)$.
This bound remains uniform under small $C^2$
perturbations of $H$.
\end{abstract}


\smallskip
\section{Introduction}
Let $(\mathbb S^n,g_0,\sigma)$ be the round sphere equipped with its
standard spin structure, where $n\geq2$. In this paper, we study
the critical nonlinear Dirac equation
\begin{equation}\label{eq:intro-dirac}
D_{g_0}\psi=H(x)|\psi|^{2^*-2}\psi
\qquad\text{on }\mathbb S^n,
\end{equation}
where $D_{g_0}$ is the Dirac operator, $H:\mathbb S^n\to\mathbb R$
is a prescribed function, and $|\cdot|$ denotes the pointwise
norm induced by the Hermitian metric on the spinor bundle
$\Sigma \mathbb S^n$. Here
\(
2^*=\frac{2n}{n-1}
\)
is the critical exponent for the Sobolev embedding
\(
H^{1/2}(\mathbb S^n,\Sigma \mathbb S^n)
\hookrightarrow L^p(\mathbb S^n,\Sigma\mathbb S^n)\),
\(1\leq p\leq2^*.
\)

When $H\equiv1$, equation~\eqref{eq:intro-dirac} is known
as the spinorial Yamabe equation. It is closely related
to the B\"ar--Hijazi--Lott invariant, defined by minimizing
the first positive Dirac eigenvalue within a conformal
class under a volume normalization \cite{Ammann2009}.
For a prescribed function $H$, the equation also has a
geometric interpretation in dimension two: the spinorial
Weierstrass representation relates its solutions to
conformal immersions, possibly with branch points,
of $\mathbb S^2$ into $\mathbb R^3$ with prescribed mean curvature
$H$ \cite{A2003,Ammann2009,Xu2024,Friedrich1998,Kenmotsu1979,KusnerSchmitt1996,Matsutani1999,Taimanov1998,Taimanov1999}.

The corresponding critical Dirac equation on $\mathbb R^n$
has also been studied extensively. Ammann, Humbert and Ould Ahmedou \cite{BE} established
a Pohozaev-type identity for nonlinear Dirac equations.
At the critical exponent, it yields a Kazdan--Warner-type
obstruction on closed spin manifolds.
Borrelli--Frank \cite{BF2020} established sharp decay
estimates and asymptotic expansions for finite-energy
solutions. Borrelli--Malchiodi--Wu \cite{BMW2021}
classified ground-state solutions as conformal images
of Killing spinors on the round sphere and obtained
their explicit Euclidean form. These results provide
basic tools for identifying and analyzing the limiting
profiles arising in blow-up analysis.

The analysis of solution sequences involves two main
difficulties: the lack of compactness of the critical
Sobolev embedding and the strong indefiniteness of the
associated variational functional.
Isobe \cite{Isobe2011} established regularity results
for weak solutions and a global compactness theorem
for the associated Palais--Smale sequences on compact
spin manifolds, describing the loss of compactness
through Euclidean bubbles.
For subcritical approximations on the sphere,
Xu \cite{Xu2024} obtained a concentration-compactness
alternative and a single bubble decomposition in a
suitable low energy range. For sequences of solutions
in this setting, the concentration point is a critical
point of the prescribed coefficient. More recently, Chen, Chen and Zhu \cite{CCZ2026} established
energy identities for nonlinear Dirac systems on fixed spin
manifolds and on non collapsed degenerating Einstein manifolds.
Their fixed manifold result includes the critical Dirac equation
with a fixed $C^1$ coefficient.
Their proof uses cylindrical analysis and a three circles
estimate to control the energy in the neck regions.

Wang and Zhang \cite{WangZhang2025} established a quantitative
stability inequality for the sharp spinorial Sobolev inequality
on $\mathbb S^n$.
Their estimate controls the distance to the family of conformal
images of Killing spinors.
Compactness cannot be expected without additional
assumptions. Isobe--Sire--Xu \cite{ISX2024} constructed
positive nonconstant $C^2$ coefficients on the round
sphere for which the equation admits sequences of
solutions with uniformly bounded critical integrals
and unbounded $L^\infty$ norms.
These examples motivate a detailed analysis of the
local structure of blow-up. In particular, one seeks
to determine when blow-up points are isolated simple
and to obtain precise estimates near such points.

We study the compactness and blow-up behavior of sequences
of solutions to \eqref{eq:intro-dirac}. More precisely, let
$\{\psi_i\}\subset H^{1/2}(\mathbb S^n,\Sigma\mathbb S^n)$ be a sequence
of weak solutions of
\begin{equation}\label{eq:intro-sequence}
D_{g_0}\psi_i=H_i(x)|\psi_i|^{2^*-2}\psi_i
\qquad\text{on }\mathbb S^n,
\end{equation}
where
\begin{equation}\label{eq:intro-assumptions}
H_i\to H\quad\text{in }C^0(\mathbb S^n),
\qquad
\int_{\mathbb S^n}|\psi_i|^{2^*}\,dv_{g_0}\leq\Lambda
\end{equation}
for some constant $\Lambda>0$.
We are mainly concerned with the behavior of $\psi_i$
when $\|\psi_i\|_{L^\infty(\mathbb S^n)}\to\infty$, and in particular
with the concentration of its critical mass and its
asymptotic profile near blow-up points.

By the conformal covariance of the Dirac operator, analysis of \eqref{eq:intro-sequence} can be reduced to a Euclidean equation.
For $p\in \mathbb S^n$, let
$\pi_p:\mathbb S^n\setminus\{-p\}\to\mathbb R^n$
be the standard stereographic projection with $\pi_p(p)=0$. We use $x$ for points on $\mathbb S^n$ and
$y=\pi_p(x)$ for their stereographic coordinates. The corresponding spinors on $\mathbb R^n$ are defined by
\[
\varphi_i(y)
:=
f(y)^{\frac{n-1}{2}}
F\bigl(\psi_i(\pi_p^{-1}(y))\bigr),
\]
where $f(y)=\frac{2}{1+|y|^2}$, and $F$ is the fiberwise isometry described in Section~\ref{Conformal covariance and stereographic projection}.
By conformal covariance,
\begin{equation}\label{eq:intro-euclidean}
D_{g_{{\mathbb R^n}}}\varphi_i
=
\widetilde H_i(y)|\varphi_i|^{2^*-2}\varphi_i
\quad\text{in }\mathbb R^n,
\qquad
\widetilde H_i:=H_i\circ\pi_p^{-1}.
\end{equation}
For sequences $\{\psi_i\}$ satisfying
\eqref{eq:intro-sequence} and \eqref{eq:intro-assumptions},
we introduce the following definitions,
adapting the terminology of \cite{Li1995JDE}.
\begin{defn}[Blow-up point]
\label{def:intro-blowup}
A point $p\in \mathbb S^n$ is called a blow-up point of
$\{\psi_i\}$ if there exist $x_i\to p$ such that
$|\psi_i(x_i)|\to\infty$.
Equivalently, in the above chart, there exist $y_i\to0$
such that $|\varphi_i(y_i)|\to\infty$.
\end{defn}

\begin{defn}[Isolated blow-up point]
\label{def:intro-isolated}
A blow-up point $p$ is called isolated if there exist
constants $C,\bar r>0$, independent of $i$, and local
maximum points $y_i\to0$ of $|\varphi_i|$ such that
$|\varphi_i(y_i)|\to\infty$ and
\[
|\varphi_i(y)|
\le C|y-y_i|^{-\frac{n-1}{2}},
\qquad
0<|y-y_i|<\bar r,
\]
for all sufficiently large $i$.
\end{defn}

\begin{defn}[Isolated simple blow-up point]
\label{def:intro-simple}
An isolated blow-up point $p$ is called isolated simple
if, for some choice of the points $y_i$ in
Definition~\ref{def:intro-isolated}, the function
\[
\overline w_i(r)
:=
\frac{r^{n-1}}{|\partial B_r|}
\int_{\partial B_r(y_i)}|\varphi_i|^2\,d\sigma
\]
has exactly one critical point in $(0,\rho)$,
for some $\rho\in(0,\bar r)$ independent of $i$
and all sufficiently large $i$.
\end{defn}
\begin{rem}Under the $C^1$ assumption on the coefficients, local elliptic regularity gives
$\varphi_i\in C^1_{\mathrm{loc}}$.
Consequently, $\overline w_i\in C^1((0,\rho))$.
The use of $|\varphi_i|^2$ avoids differentiating
$|\varphi_i|$ at its zeros.
\end{rem}
We first describe concentration at the first threshold for
uniformly convergent coefficients.
This provides the starting point for our analysis of isolated
simplicity and the compactness criterion proved below.
\begin{thm}
\label{thm:intro-threshold}
Let $n\ge3$ and let $\{\psi_i\}$ satisfy
\eqref{eq:intro-sequence} and \eqref{eq:intro-assumptions}.
Assume that $H_i\ge h_0>0$ and
\begin{equation}\label{eq:intro-energy-threshold}
\limsup_{i\to\infty}\int_{\mathbb S^n}|\psi_i|^{2^*}\,dv_{g_0}
\le
H_{\max}^{-n}\left(\frac n2\right)^n\omega_n,
\end{equation}
where $H_{\max}:=\max\limits_{\mathbb S^n}H$ and
$\omega_n:=|\mathbb S^n|$.
If $\|\psi_i\|_{L^\infty(\mathbb S^n)}\to\infty$,
then, after passing to a subsequence, the following hold.

\begin{enumerate}
\item[(i)]
There is a unique blow-up point $p\in\mathbb S^n$, and
\(
H(p)=H_{\max}\),
\(
|\psi_i|^{2^*}\,dv_{g_0}
\rightharpoonup
H_{\max}^{-n}\left(\frac n2\right)^n\omega_n\,\delta_p
\)
as finite Radon measures.

\item[(ii)]
Let $\varphi_i$ be the corresponding spinors on
$\mathbb R^n$ in the chart centered at $p$.
There exist local maximum points $y_i\to0$ of
$|\varphi_i|$, with $|\varphi_i(y_i)|\to\infty$,
such that, setting
\[
\lambda_i:=|\varphi_i(y_i)|^{-\frac{2}{n-1}},
\qquad
\Phi_i(z):=
\lambda_i^{\frac{n-1}{2}}\varphi_i(y_i+\lambda_i z),
\]we have
\(
\Phi_i(z)\longrightarrow\Phi(z)\)
in \(C^0_{\mathrm{loc}}(\mathbb R^n)
\) and \(L^{2^*}(\mathbb R^n),
\)
where
\[
\Phi(z)
=
H(p)^{-\frac{n-1}{2}}\sqrt{\lambda}
\left(\frac{2\lambda}{\lambda^2+|z|^2}\right)^{\frac n2}
\left(1-\frac{z}{\lambda}\right)\cdot_{g_{\Rn}}\widetilde{\Phi}_0,
\qquad
\lambda=\frac{n}{H(p)},
\]
for some constant spinor $\widetilde{\Phi}_0\in\mathbb C^{2^{[ \frac{n}{2}]}}$
satisfying
\(
|\widetilde{\Phi}_0|
=
\frac{1}{\sqrt2}
\left(\frac n2\right)^{\frac{n-1}{2}}.
\)
\end{enumerate}
\end{thm}
Our second main result concerns the isolated simplicity
of the blow-up point.

\begin{thm}
\label{thm:intro-simple}
Under the hypotheses of Theorem~\ref{thm:intro-threshold},
let $p$ be the unique blow-up point
obtained there. If
\(
\sup\limits_i\|H_i\|_{C^1(U)}<\infty
\)
for some neighborhood $U$ of $p$, then $p$ is an
isolated simple blow-up point.
\end{thm}
Our final result gives compactness for solutions with critical
mass slightly above the first threshold.
The estimate is uniform under small $C^2$ perturbations of
the coefficient. We use the convention
$\Delta_{g_0}=\operatorname{div}_{g_0}\nabla_{g_0}$.
\begin{thm}\label{thm:intro-compactness}
Let $n\ge3$ and let $H\in C^2(\mathbb S^n)$ be positive.
Assume that
\(
\Delta_{g_0}H(p)<0\)
whenever \(
H(p)=\max_{\mathbb S^n}H.
\)
Then there exist $\varepsilon,\delta,C>0$, depending only
on $H$ and $n$, such that, for every positive
$\widehat H\in C^2(\mathbb S^n)$ with
$\|\widehat H-H\|_{C^2(\mathbb S^n)}<\delta$,
every weak solution
$\psi\in H^{1/2}(\mathbb S^n,\Sigma\mathbb S^n)$ of
\[
D_{g_0}\psi=\widehat H|\psi|^{2^*-2}\psi
\qquad\text{on }\mathbb S^n
\]
satisfying
\[
\int_{\mathbb S^n}|\psi|^{2^*}\,dv_{g_0}
\le \widehat H_{\max}^{-n}
\left(\frac n2\right)^n\omega_n+\varepsilon
\]
has $\|\psi\|_{L^\infty(\mathbb S^n)}\le C$.
\end{thm}
The concentration result at the first threshold is 
related to the energy identities and stability results
discussed above. Our main contribution is a precise description of the
solutions near the blow-up point.
We prove isolated simplicity and obtain sharp pointwise,
gradient, and Green-type asymptotics.
Under local $C^2$ convergence of the coefficients, we show
that the Hessian of the limiting coefficient vanishes at
the blow-up point.
This necessary condition yields compactness near the first
threshold, uniformly under small $C^2$ perturbations of
the coefficient. A natural next step is to study higher energy levels,
where several bubbles and their interactions may occur.

In the scalar setting, blow-up analysis has played a central
role in compactness theory since the work of
Schoen \cite{Schoen1989}.
Caffarelli, Gidas, and Spruck \cite{CGS1989} classified
the positive entire solutions of the critical Euclidean
equation, which give the standard bubble profiles.
For the prescribed scalar curvature problem,
Li \cite{Li1995JDE} developed a detailed analysis of isolated
and isolated simple blow-up points.
His argument combines sharp local estimates, rescaling
at a second critical point of the weighted spherical average,
and a Pohozaev-type identity to prove simplicity under
suitable assumptions on the coefficient.
Further applications of blow-up analysis to compactness
for Yamabe-type equations can be found, for example, in
\cite{LiZhu1999,KMS2009}.

 Isolation and simplicity describe the local shape of a
concentrating bubble. As its peak becomes higher and narrower,
isolation gives a uniform bound on the size of the solution
in terms of the distance from its center. Simplicity requires
the weighted spherical average to increase and then decrease,
as it does for a standard bubble, with no further critical
point in a fixed small neighborhood. Convergence to a bubble
after rescaling describes the solution near its peak.
Isolated simplicity gives additional control of its radial
behavior farther from the center.

For the nonlinear Dirac equation, the maximum principle and Harnack inequalities for positive
scalar solutions do not apply directly.
Our basic tools are small-energy regularity, interior elliptic
estimates, and the conformal covariance of the Dirac operator.
Under the energy bound \eqref{eq:intro-energy-threshold},
the first bubble exhausts the total energy, small-energy regularity then yields isolation.
If simplicity fails, the weighted spherical average
$\bar w_i$ has a second critical point $\mu_i$ satisfying
\[
\mu_i\to0,
\qquad
\frac{\mu_i}{\lambda_i}\to\infty,
\qquad
\bar w_i'(\mu_i)=0.
\]
We apply the Kelvin transform to the rescaled solutions
$\Phi_i$. The smallness of the energy tails gives uniform
compactness near the origin. This determines the normalized
limit after rescaling by $\mu_i$.
A scalar identity for the squared norm then yields
$C^1$ convergence of the corresponding normalized weighted
spherical averages.
Their derivatives vanish at $r=1$ because
$\bar w_i'(\mu_i)=0$.
However, their limit is a positive multiple of $r^{1-n}$
and has a strictly negative derivative there.
This contradiction proves isolated simplicity without
using a Pohozaev-type identity. For the compactness result, we combine the sharp
decay estimate and the Green convergence with the
Pohozaev identity of Ammann, Humbert and Ould Ahmedou
\cite{BE}. This gives a second-order condition at the
blow-up point and excludes blow-up under the assumption
of Theorem~\ref{thm:intro-compactness}.

The paper is organized as follows. Section~\ref{Preliminaries} collects
the geometric and analytic preliminaries. In Section~\ref{Compactness alternative},
we establish small-energy regularity and a compactness
alternative, and construct the rescaled limits.
Sections~\ref{Proof of Theorem 1.1} and~\ref{Proof of Theorem 1.2} prove Theorems~\ref{thm:intro-threshold} and~\ref{thm:intro-simple}, respectively.
In Section~\ref{Proof of Theorem 1.3}, we derive sharp decay estimates and Green
convergence, and use the Pohozaev identity to prove
Theorem~\ref{thm:intro-compactness}.

\section{Preliminaries}\label{Preliminaries}
\subsection{Spin structures and the Dirac operator}
\label{subsec:spin-geometry}

In this subsection, we briefly recall some standard facts about spin
geometry and the Dirac operator. We refer to
\cite{A2003,BHMMM2015,F2000,LM1989}
for further background. Let $(M^n,g,\sigma)$ be an $n$-dimensional spin manifold with $n\ge 2$, where $g$ is a Riemannian metric and $\sigma : P_{Spin}(M)\to P_{SO}(M)$ is a fixed spin structure. Let
$
\rho:Spin(n)\to Aut(\Sigma_n)
$
be the complex spin representation, and let
$$
\Sigma M:=P_{Spin}(M)\times_\rho \Sigma_n
$$
be the associated complex spinor bundle. The bundle $\Sigma M$ carries a natural spinorial Levi--Civita connection $\nabla$, a pointwise Hermitian scalar product $\langle\cdot,\cdot\rangle_g$, and a Clifford multiplication
$
\cdot_g:TM\otimes \Sigma M\to \Sigma M$.
Clifford multiplication satisfies
\[
X\cdot_g Y\cdot_g\psi+Y\cdot_g X\cdot_g\psi
=-2g(X,Y)\psi.
\]
These structures are compatible in the sense that
\[
\langle X\cdot_g\varphi,\psi\rangle_g=-\langle\varphi,X\cdot_g\psi\rangle_g,
\]
\[
X\langle\varphi,\psi\rangle_g=\langle\nabla_X\varphi,\psi\rangle_g+\langle\varphi,\nabla_X\psi\rangle_g,
\]
and
\[
\nabla_X (Y\cdot_g\varphi)=(\nabla_XY)\cdot_g\varphi+Y\cdot_g\nabla_X\varphi
\]
for all $X,Y\in\Gamma(TM)$ and $\varphi,\psi\in\Gamma(\Sigma M)$. We write
\(
(\varphi,\psi)_g:=\operatorname{Re}\langle\varphi,\psi\rangle_g\),
\(
|\psi|_g^2:=\langle\psi,\psi\rangle_g.
\)

If \(\{e_1,\ldots,e_n\}\) is a local oriented orthonormal frame of
\(TM\), the Dirac operator is defined by
$$
D_g\psi
:=
\sum_{j=1}^n
e_j\cdot_g\nabla_{e_j}\psi.
$$
The Dirac operator is a first-order elliptic operator
with the following standard properties; see \cite{LM1989}.

\begin{enumerate}
\item[(i)] \emph{Formal self-adjointness.}
For $\varphi,\psi\in C_c^\infty(M,\Sigma M)$,
\[
\int_M\langle D_g\varphi,\psi\rangle_g\,dv_g
=
\int_M\langle\varphi,D_g\psi\rangle_g\,dv_g.
\]

\item[(ii)] \emph{Leibniz rule.}
For $f\in C^\infty(M,\mathbb R)$ and
$\psi\in C^\infty(M,\Sigma M)$,
\[
D_g(f\psi)=\nabla^g f\cdot_g\psi+fD_g\psi.
\]

\item[(iii)] \emph{Schr\"odinger--Lichnerowicz formula.}
\[
D_g^2=\nabla^*\nabla+\frac{\mathrm{Scal}_g}{4},
\]
where $\nabla^*\nabla$ is the connection Laplacian
and $\mathrm{Scal}_g$ is the scalar curvature.
\end{enumerate}
\subsection{Conformal covariance }\label{Conformal covariance and stereographic projection}

We recall the conformal covariance of the Dirac operator;
see \cite{Xu2024}.
Let $\bar g=f^2g$, where $f$ is a smooth positive
function on $M$. There is a fiberwise isometry
\(
\beta:\Sigma(M,g)\longrightarrow\Sigma(M,\bar g)
\)
such that
\begin{equation}
\label{eq:conformal-dirac}
D_{\bar g}
\bigl(f^{-\frac{n-1}{2}}\beta(\psi)\bigr)
=
f^{-\frac{n+1}{2}}\beta(D_g\psi).
\end{equation}

Let $p\in\mathbb S^n$, and let
\(
\pi_p:\mathbb S^n\setminus\{-p\}\longrightarrow\mathbb R^n\),
\(\pi_p(p)=0,
\)
be the stereographic projection from the antipodal
point $-p$. Then
\(
(\pi_p^{-1})^*g_0=f(y)^2g_{\mathbb R^n}\),
\(
f(y)=\frac{2}{1+|y|^2}.
\)
For $0<\rho<\pi$, set
\(
U_\rho(p):=\pi_p\bigl(B_\rho^{g_0}(p)\bigr)
=B_{\tan(\rho/2)}(0).
\)
When $p$ is fixed, we simply write $U_\rho$.
Let
\(
F:\Sigma(\mathbb R^n,f^2g_{\mathbb R^n})
\longrightarrow\Sigma(\mathbb R^n,g_{\mathbb R^n})
\)
be the corresponding fiberwise isometry.
For a spinor $\psi$ on $B_\rho^{g_0}(p)$, define
\(
\varphi(y)
=
f(y)^{\frac{n-1}{2}}
F\bigl(\psi\circ\pi_p^{-1}(y)\bigr)\),
\(y\in U_\rho.
\)

By \eqref{eq:conformal-dirac}, if
\[
D_{g_0}\psi=H|\psi|^{2^*-2}\psi
\qquad\text{in }B_\rho^{g_0}(p),
\]
then
\begin{equation}\label{eq:euclidean-dirac}
D_{g_{{\mathbb R^n}}}\varphi
=\widetilde H(y)|\varphi|^{2^*-2}\varphi
\qquad\text{in }U_\rho,
\end{equation}
where
\(
\widetilde H(y):=H\bigl(\pi_p^{-1}(y)\bigr).
\)
Since $F$ is a fiberwise isometry and
$(\pi_p^{-1})^*dv_{g_0}=f^n\,dy$, we also have
\begin{equation}
\label{eq:conformal-critical-integral}
\int_{U_\rho}|\varphi|^{2^*}\,dy
=
\int_{B_\rho^{g_0}(p)}|\psi|^{2^*}\,dv_{g_0}.
\end{equation}

We next pass from the Euclidean coordinates above to cylindrical
coordinates. For $y_0\in\mathbb R^n$ and $R>0$, write
\[
y=y_0+e^t\theta,
\qquad
(t,\theta)\in(-\infty,\log R)\times\mathbb S^{n-1}.
\]
Then
\[
g_{\mathbb R^n}=e^{2t}g_{\mathrm{cyl}},
\qquad
g_{\mathrm{cyl}}=dt^2+g_{\mathbb S^{n-1}}.
\]
Let
$\beta:\Sigma_{g_{\mathbb R^n}}\to\Sigma_{g_{\mathrm{cyl}}}$
be the corresponding fibrewise isometry.
For a spinor $\varphi$ on $B_R(y_0)\setminus\{y_0\}$, set
\begin{equation}\label{eq:cyl-transform}
U(t,\theta)
=e^{\frac{n-1}{2}t}
\beta\bigl(\varphi(y_0+e^t\theta)\bigr).
\end{equation}
Conformal covariance gives
\begin{equation}\label{eq:cyl-covariance}
D_{\mathrm{cyl}}U(t,\theta)
=e^{\frac{n+1}{2}t}
\beta\bigl((D_{{g_{{\mathbb R^n}}}}\varphi)(y_0+e^t\theta)\bigr).
\end{equation}
Following the product
decomposition of Ammann and Jammes
\cite{AmmannJammes}:
\begin{equation}\label{eq:cyl-decomposition}
D_{\mathrm{cyl}}=D_{\mathrm{vert}}+D_{\mathrm{hor}},
\qquad
D_{\mathrm{vert}}
=\sum_{\alpha=1}^{n-1}e_\alpha\cdot\nabla_{e_\alpha},
\qquad
D_{\mathrm{hor}}=\partial_t\cdot\nabla_{\partial_t},
\end{equation}
where $(e_1,\ldots,e_{n-1})$ is a local orthonormal frame of $\mathbb S^{n-1}$. The connection and Clifford multiplication
are those of $g_{\mathrm{cyl}}$, with the spin structure
induced from Euclidean space.
The operators $D_{\mathrm{vert}}$ and $D_{\mathrm{hor}}$
anticommute.
 Moreover, $D_{\mathrm{vert}}$ is independent
of $t$ and self-adjoint on each slice $\{t\}\times\mathbb S^{n-1}$. Since $\nabla_{e_\alpha}\partial_t=0$ and
$g_{\mathrm{cyl}}(e_\alpha,\partial_t)=0$
for $\alpha=1,\ldots,n-1$, the Clifford relations give
\[
D_{\mathrm{vert}}(\partial_t\cdot U)
=-\partial_t\cdot D_{\mathrm{vert}}U.
\]
Consequently, $\partial_t\cdot D_{\mathrm{vert}}$
is also self-adjoint. Using the Dirac spectrum of the round sphere
\cite{Baer1996}, we obtain
\[
\operatorname{spec}(D_{\mathrm{vert}})
=
\left\{
\pm\left(\frac{n-1}{2}+k\right):
k=0,1,\ldots
\right\}.
\]
In particular,
\begin{equation}\label{eq:cyl-gap}
\|D_{\mathrm{vert}}U\|_{L^2(\mathbb S^{n-1})}
\ge
\frac{n-1}{2}\|U\|_{L^2(\mathbb S^{n-1})}
\end{equation}
for every smooth section $U$ of the spinor bundle
restricted to a slice.
\subsection{Function spaces and elliptic estimates}
Let $(M^n,g)$ be a closed Riemannian spin manifold,
$n\geq2$. Following \cite{Isobe2011,Isobe2013},
we recall the spectral definition of fractional
Sobolev spaces of spinors. The self-adjoint realization of $D_g$ on
$L^2(M,\Sigma M)$ has compact resolvent.
Thus, there exists an
\(L^2\)-orthonormal basis of eigenspinors
\(\{\varphi_k\}_{k\geq1}\) satisfying
$D_g\varphi_k=\lambda_k\varphi_k$ and $|\lambda_k|\to\infty$ as $k\to\infty$.
For $s \geq 0$,  define 
$$
\left|D_g\right|^s \varphi=\sum_{k=1}^{\infty}\left|\lambda_k\right|^s \alpha_k \varphi_k,
$$
for $\varphi=\sum_{k=1}^{\infty} \alpha_k \varphi_k \in L^2(M, \Sigma M)$. Denoting the domain of $\left|D_g\right|^s$ by $H^s(M, \Sigma M)$, we see that $\varphi\in H^s(M, \Sigma M)$ if and only if $\sum_{k=1}^{\infty}\left|\lambda_k\right|^{2 s}\left|\alpha_k\right|^2<\infty$. Hence $H^s(M, \Sigma M)$ coincides with the usual Sobolev space of order $s$, that is $W^{s, 2}(M, \Sigma M)$; see \cite{A2003}. We equip $H^s(M,\Sigma M)$ with the inner product
$$
\left( \varphi,\psi\right)_{s,2}
:=\bigl(|D_g|^s\varphi,|D_g|^s\psi\bigr)_{L^2}
+
(\varphi,\psi)_{L^2},
$$
where $(\cdot,\cdot)_{L^2}$ denotes the $L^2$-inner product on spinors, and write $\|\cdot\|_{H^s}$ for the induced norm. 
Thus, for
\(\varphi=\sum_{k=1}^{\infty}\alpha_k\varphi_k\), we have
$$
\|\varphi\|_{H^s}^2
=
\sum_{k=1}^{\infty}
\bigl(1+|\lambda_k|^{2s}\bigr)|\alpha_k|^2.
$$ We use $L^p$, $W^{k,p}$, and $C^{k,\alpha}$ for the
usual Lebesgue, Sobolev, and H\"older spaces of
functions or spinors, respectively.
The underlying domain and bundle are omitted
when clear from the context. Let $2^*=\frac{2 n}{n-1}$ be the critical Sobolev exponent for the embedding $H^{1 / 2}(M, \Sigma M) \hookrightarrow L^q(M, \Sigma M)$, $1 \leq q \leq 2^*$.

We then introduce the corresponding Euclidean configuration space. On $\mathbb{R}^n$, we consider the fractional space
\[
\mathcal D^{1/2}(\mathbb R^n,\Sigma\mathbb R^n)
=
\left\{
\varphi\in L^{2^*}(\mathbb R^n,\Sigma\mathbb R^n):
|D_{g_{\mathbb{R}^n}}|^{1/2}\varphi\in L^2
\right\}
\]
endowed with the norm
$$
\|\varphi\|_{\mathcal D^{1/2}}
:=
\bigl\||D_{g_{\mathbb{R}^n}}|^{1/2}\varphi\bigr\|_2+\|\varphi\|_{2^*}.
$$
Here $D_{g_{\mathbb{R}^n}}$ denotes the Euclidean Dirac operator, and $|D_{g_{\mathbb{R}^n}}|^{1/2}$ is defined through the Fourier transform by
$
\mathscr{F}\bigl(|D_{g_{\mathbb{R}^n}}|^{1/2}\psi\bigr)(\xi)
=
|\xi|^{1/2}\mathscr{F}(\psi)(\xi)$ and $\|\left|D_{g_{\mathbb{R}^n}}\right|^{1 / 2} \psi\|_2=\||\cdot|^{1 / 2} \mathscr{F}(\psi)\|_2$. Notice that $\mathcal D^{1/2}(\mathbb R^n,\Sigma\mathbb R^n)$ is isomorphic to  
$H^{1/2}(\mathbb S^n,\Sigma\mathbb S^n)$ via the stereographic projection.

Finally, we recall the interior estimates for the
Dirac operator on a smooth Riemannian spin manifold; see
\cite[Appendix A]{Ammann2009}.
Let $\Omega',\Omega\subset M$ be open sets such that
$\overline{\Omega'}$ is compact and contained in $\Omega$.
For $k\in\mathbb N_0$ and $1<p<\infty$,
\[
\|\psi\|_{W^{k+1,p}(\Omega')}
\le C\left(
\|D_g\psi\|_{W^{k,p}(\Omega)}
+\|\psi\|_{L^p(\Omega)}
\right),
\]
where $C$ is independent of $\psi$. If $\psi\in L^p(\Omega)$ and
$D_g\psi\in W^{k,p}(\Omega)$ in the sense of
distributions, then
$\psi\in W^{k+1,p}(\Omega')$ and the above
estimate holds.
Analogous interior Schauder estimates also hold.

\section{Compactness alternative}\label{Compactness alternative}
\subsection{Local \texorpdfstring{$\varepsilon$}{epsilon}-regularity}

We first prove a local $\varepsilon$-regularity result for the critical nonlinear
Dirac equation on $\Rn$.
 \begin{lem}[Local $\varepsilon$-regularity]\label{lem:small-energy-regularity-Dirac}
Let $n\geq3$, and let
$\varphi\in L^{2^*}(B_2,\Sigma\mathbb R^n)$
be a distributional solution of
\[
D_{g_{\R^n}}\varphi
=
\widetilde H(y)|\varphi|^{2^*-2}\varphi
\quad\text{in }B_2,
\qquad
\|\widetilde H\|_{L^\infty(B_2)}\leq A.
\]
There exists $\varepsilon_0=\varepsilon_0(n,A)>0$
such that, if
\[
\int_{B_2}|\varphi|^{2^*}\,dy\leq\varepsilon_0,
\] then, for every $1<q<\infty$,
\[
\|\varphi\|_{L^\infty(B_1)}
\leq C\|\varphi\|_{L^{2^*}(B_2)},\qquad
\|\varphi\|_{W^{1,q}(B_1)}
\leq C_q\|\varphi\|_{L^{2^*}(B_2)},
\]
where $C=C(n,A)$ and $C_q=C(n,A,q)$.
\end{lem}
\begin{proof}
We take $\varepsilon_0\leq1$ .
Since $n(2^*-2)=2^*$,
\begin{equation}
\label{eq:eps-small-coefficient}
\bigl\|\widetilde H|\varphi|^{2^*-2}\bigr\|_{L^n(B_2)}
\leq A\varepsilon_0^{1/n}.
\end{equation}
We first improve the integrability of $\varphi$.
Set $q_0=2n/(n-2)>2^*$.
Choose $\eta\in C_c^\infty(B_2)$ with $0\leq\eta\leq1$,
$\eta=1$ on $B_{7/4}$, and $|\nabla\eta|\leq C$.
The product rule gives
\[
D_{g_{\R^n}}(\eta\varphi)
=\widetilde H|\varphi|^{2^*-2}\eta\varphi+\nabla\eta\cdot\varphi
\]
in the sense of distributions on $\mathbb R^n$, where all products are
extended by zero. Below, the coefficient
$\widetilde H|\varphi|^{2^*-2}$ is also extended by zero outside $B_2$.

Let $\Gamma$ be the Euclidean Dirac fundamental solution; see
\cite{BF2020}. With our Clifford convention,
\[
\Gamma(y)=-\frac{1}{\omega_{n-1}}\frac{y}{|y|^n}\cdot,
\qquad D_{g_{\R^n}}\Gamma=\delta_0\operatorname{Id},
\qquad \omega_{n-1}=|\mathbb S^{n-1}|.
\]
For every compactly supported spinor distribution $u$, one has
$\Gamma*(D_{g_{\R^n}}u)=u$.
Indeed, if $G$ is the scalar fundamental solution of $-\Delta$, then
$\Gamma=D_{g_{\R^n}}G$, and distributional convolution gives
\[
\Gamma*(D_{g_{\R^n}}u)
=G*(D_{g_{\R^n}}^{\,2}u)
=G*(-\Delta u)=(-\Delta G)*u=u.
\]
Applying this identity to $u=\eta\varphi$ yields
\begin{equation}
\label{eq:eps-integral-equation}
\eta\varphi
=\Gamma*\bigl(\widetilde H|\varphi|^{2^*-2}\eta\varphi\bigr)
 +\Gamma*(\nabla\eta\cdot\varphi).
\end{equation}
We have
\[
\bigl\|\widetilde H|\varphi|^{2^*-2}\eta\varphi
\bigr\|_{L^{\frac{2n}{n+1}}(\mathbb R^n)}
\leq A\|\varphi\|_{L^{2^*}(B_2)}^{2^*-1}.
\]
Moreover, by H\"older's inequality,
\[
\|\nabla\eta\cdot\varphi\|_{L^{\frac{2n}{n+1}}(\mathbb R^n)}
\leq \|\nabla\eta\|_{\infty}|B_2|^{1/n}
\|\varphi\|_{L^{2^*}(B_2)}.
\]
Thus both source terms belong to $L^{\frac{2n}{n+1}}(\mathbb R^n)$.
Since $|\Gamma(y)|\leq C_n|y|^{1-n}$ and
$\frac{1}{2^*}=\frac{n+1}{2n}-\frac{1}{n}$, the Hardy--Littlewood--Sobolev
inequality gives
\[
\|\Gamma*F\|_{L^{2^*}(\mathbb R^n)}
\leq C_n\|F\|_{L^{\frac{2n}{n+1}}(\mathbb R^n)}.
\]
Consequently, both convolution terms belong to
$L^{2^*}(\mathbb R^n)$.

For $q\in\{2^*,q_0\}$, the Hardy--Littlewood--Sobolev and H\"older
inequalities imply
\begin{equation}
\label{eq:eps-contraction-estimate}
\begin{aligned}
\bigl\|\Gamma*(\widetilde H|\varphi|^{2^*-2}v)\bigr\|_{L^q}
&\leq C_{n,q}
 \bigl\|\widetilde H|\varphi|^{2^*-2}v\bigr\|_{L^{nq/(n+q)}}\\
&\leq C_{n,q}A\varepsilon_0^{1/n}\|v\|_{L^q}.
\end{aligned}
\end{equation}
Choose $\varepsilon_0=\varepsilon_0(n,A)>0$ so that the last operator
norm is at most $1/2$ for both exponents.
Since $2^*>2$ and $\nabla\eta$ is bounded with compact support,
\[
\nabla\eta\cdot\varphi\in L^{2n/(n+1)}\cap L^2,
\qquad
\|\nabla\eta\cdot\varphi\|_{L^2}\leq C\|\varphi\|_{L^{2^*}(B_2)}.
\]
Consequently,
\[
\Gamma*(\nabla\eta\cdot\varphi)\in L^{2^*}\cap L^{q_0},
\qquad
\|\Gamma*(\nabla\eta\cdot\varphi)\|_{L^{q_0}}\leq C\|\varphi\|_{L^{2^*}(B_2)}.
\]
Thus the affine map
\[
v\longmapsto
\Gamma*(\widetilde H|\varphi|^{2^*-2}v)
+\Gamma*(\nabla\eta\cdot\varphi)
\]
is a contraction on the Banach space $L^{2^*}\cap L^{q_0}$,
equipped with the sum of its two norms.
Its fixed point is also a fixed point in $L^{2^*}$, where
\eqref{eq:eps-contraction-estimate} gives uniqueness.
By \eqref{eq:eps-integral-equation}, this fixed point is $\eta\varphi$.
In particular, $\eta\varphi\in L^{q_0}$; estimating its fixed-point
equation in $L^{q_0}$ now gives
\begin{equation}
\label{eq:eps-local-integrability-gain}
\|\varphi\|_{L^{q_0}(B_{7/4})}\leq C\|\varphi\|_{L^{2^*}(B_2)}.
\end{equation}
We next use interior elliptic estimates and Sobolev
embedding to obtain an $L^\infty$ bound.
If $\|\varphi\|_{L^q(B_R)}\leq C_q \|\varphi\|_{L^{2^*}(B_2)}$, then the equation gives
\[
\|D_{g_{\R^n}}\varphi\|_{L^{q/(2^*-1)}(B_R)}
\leq A\|\varphi\|_{L^q(B_R)}^{2^*-1}
\leq C \|\varphi\|_{L^{2^*}(B_2)},
\]
because $\|\varphi\|_{L^{2^*}(B_2)}\leq1$. For the exponents used below, $q/(2^*-1)>1$;
also $\|\varphi\|_{L^{q/(2^*-1)}(B_R)}\leq C\|\varphi\|_{L^{2^*}(B_2)}$ by finite-measure inclusion.
Interior elliptic regularity therefore gives
$\|\varphi\|_{W^{1,q/(2^*-1)}(B_r)}\leq C\|\varphi\|_{L^{2^*}(B_2)}$ for $r<R$.
Whenever $q/(2^*-1)<n$, Sobolev embedding yields a new exponent $q'$ with
\(
\frac1{q'}=\frac {2^*-1}{q} -\frac1n.
\)
Starting at $q_0=\frac{2n}{n-2}$, these exponents increase. As long as this
recursion is used, they satisfy
\[
\frac1{q_k}=\frac{n-1}{2n}-\frac{(2^*-1)^k}{2n}.
\]
Hence after finitely many steps, depending only on $n$, one reaches
$q_k/(2^*-1)\geq n$. If equality occurs, the interior $W^{1,n}$ estimate
and the embedding $W^{1,n}\hookrightarrow L^q$ for every finite $q$
give an exponent strictly larger than $n(2^*-1)$ on a smaller ball.
Taking a fixed finite sequence of concentric balls between
$B_{7/4}$ and $B_{3/2}$, we thus obtain some $q>n(2^*-1)$ such that
\begin{equation}
\label{eq:eps-high-integrability}
\|\varphi\|_{L^q(B_{3/2})}\leq C\|\varphi\|_{L^{2^*}(B_2)}.
\end{equation}
All constants and exponents chosen here depend only on $n$ and $A$;
no further smallness condition is imposed.
The interior $W^{1,q/(2^*-1)}$ estimate on $B_{5/4}$ and Morrey's embedding
now imply
\begin{equation}
\label{eq:eps-supremum-estimate}
\|\varphi\|_{L^\infty(B_{5/4})}\leq C\|\varphi\|_{L^{2^*}(B_2)}.
\end{equation}
For every $1<q<\infty$, \eqref{eq:eps-supremum-estimate} gives
\[
\|\varphi\|_{L^q(B_{5/4})}
+\|\widetilde H|\varphi|^{2^*-2}\varphi\|_{L^q(B_{5/4})}
\leq C_q(\|\varphi\|_{L^{2^*}(B_2)}+\|\varphi\|_{L^{2^*}(B_2)}^{2^*-1})\leq C_q \|\varphi\|_{L^{2^*}(B_2)}.
\]
The interior Dirac estimate therefore yields
\[
\|\varphi\|_{W^{1,q}(B_1)}\leq C_q \|\varphi\|_{L^{2^*}(B_2)}.
\]Moreover, Morrey's embedding also gives
$\varphi\in C^{0,\gamma}(B_1)$ for every $0<\gamma<1$.
\end{proof}
By scaling, Lemma~\ref{lem:small-energy-regularity-Dirac}  yields the following
estimate on arbitrary balls.
\begin{cor}
\label{cor:epsilon-regularity-scale}
Let \(B_{2r}(y_0)\subset \mathbb R^n\), and let $
\varphi\in L^{2^*}(B_{2r}(y_0),\Sigma\mathbb R^n)$ 
be a distributional solution of
\[
D_{g_{\R^n}}\varphi
=
\widetilde H(y)|\varphi|^{2^*-2}\varphi
\qquad\text{in }B_{2r}(y_0), 
\qquad
\|\widetilde H\|_{L^\infty(B_{2r}(y_0))}\leq A.
\]
There exists $\varepsilon_0=\varepsilon_0(n,A)>0$
such that, if
\[
\int_{B_{2r}(y_0)}|\varphi|^{2^*}\,dy\leq\varepsilon_0,
\] then
\[
\sup_{B_r(y_0)}|\varphi|
\leq
C r^{-\frac{n-1}{2}}
\left(
\int_{B_{2r}(y_0)}|\varphi|^{2^*}\,dx
\right)^{1/2^*},
\]where $C=C(n,A)$.
\end{cor}
\begin{proof}
Set
\(
\varphi_r(z)
=
r^{\frac{n-1}{2}}\varphi(y_0+rz)\),
\(z\in B_2.
\)
Then
\[
D_{g_{{\mathbb R^n}}}\varphi_r
=
\widetilde H(y_0+rz)
|\varphi_r|^{2^*-2}\varphi_r,
\qquad
\|\varphi_r\|_{L^{2^*}(B_2)}
=
\|\varphi\|_{L^{2^*}(B_{2r}(y_0))}.
\]
The conclusion follows from Lemma~\ref{lem:small-energy-regularity-Dirac}  and
\(
\|\varphi_r\|_{L^\infty(B_1)}
=
r^{\frac{n-1}{2}}
\|\varphi\|_{L^\infty(B_r(y_0))}.
\)
\end{proof}
Conformal covariance gives the corresponding
local estimates on $\mathbb S^n$.
\begin{lem}\label{lem:small-energy-regularity-sphere}
Let $n\geq3$, $p\in\mathbb S^n$, and
$0<\rho<\rho_0$, where $\rho_0\in(0,1)$ is fixed.
Let
$\psi\in L^{2^*}(B_{2\rho}^{g_0}(p),\Sigma\mathbb S^n)$
be a distributional solution of
\[
D_{g_0}\psi
=
H(x)|\psi|^{2^*-2}\psi
\quad\text{in }B_{2\rho}^{g_0}(p),
\qquad
\|H\|_{L^\infty(B_{2\rho}^{g_0}(p))}\leq A.
\]
Let $\varepsilon_0$ be as in Lemma~\ref{lem:small-energy-regularity-Dirac}.
If
\[
\int_{B_{2\rho}^{g_0}(p)}
|\psi|^{2^*}\,dv_{g_0}\leq\varepsilon_0,
\]
then, for every $1<q<\infty$,
\[
\|\psi\|_{L^\infty(B_\rho^{g_0}(p))}
\leq
C\rho^{-\frac{n-1}{2}}
\|\psi\|_{L^{2^*}(B_{2\rho}^{g_0}(p))},\qquad\|\psi\|_{W^{1,q}(B_\rho^{g_0}(p))}
\leq
C_q\rho^{-1-\frac{n-1}{2}+\frac nq}
\|\psi\|_{L^{2^*}(B_{2\rho}^{g_0}(p))},
\]
where $C=C(n,A)$ and $C_q=C(n,A,q)$.
\end{lem}
\begin{proof}
Let $\varphi$ be associated with $\psi$ by the
conformal transformation in Section~\ref{Conformal covariance and stereographic projection}. Then $\varphi$ satisfies \eqref{eq:euclidean-dirac}
in $B_{\tan\rho}$
with coefficient $\widetilde H=H\circ\pi_p^{-1}$, and
\[
\|\widetilde H\|_{L^\infty(B_{\tan\rho})}\leq A,
\qquad
\|\varphi\|_{L^{2^*}(B_{\tan\rho})}
=
\|\psi\|_{L^{2^*}(B_{2\rho}^{g_0}(p))}.
\]
Since
\[
\pi_p(B_\rho^{g_0}(p))
=
B_{\tan(\rho/2)}
\subset B_{(\tan\rho)/2},
\qquad
\rho\leq\tan\rho\leq C\rho,
\]
Corollary~\ref{cor:epsilon-regularity-scale} and the uniform bounds on $f$ and
$f^{-1}$ give the $L^\infty$ estimate.

Moreover, $f$ and $f^{-1}$ have uniformly bounded
$C^1$ norms on these balls.
The conformal identification and the rescaled
$W^{1,q}$ estimate in Lemma~\ref{lem:small-energy-regularity-Dirac}  therefore yield
\[
\begin{aligned}
\|\psi\|_{W^{1,q}(B_\rho^{g_0}(p))}
&\leq
C_q\|\varphi\|_{W^{1,q}(B_{\tan(\rho/2)})}\\
&\leq
C_q\rho^{-1-\frac{n-1}{2}+\frac nq}
\|\varphi\|_{L^{2^*}(B_{\tan\rho})}.
\end{aligned}
\]
This proves the second estimate.
\end{proof}
\subsection{Concentration set}

Throughout this subsection, let $\{\psi_i\}$ be a sequence
of distributional solutions of \eqref{eq:intro-sequence} and 
satisfying \eqref{eq:intro-assumptions}.
After passing to a subsequence,
\[
\mu_i:=|\psi_i|^{2^*}\,dv_{g_0}
\rightharpoonup\mu,
\qquad
\mu(\mathbb S^n)\leq\Lambda,
\]
weakly as finite Radon measures.
Let $\varepsilon_0$ be as in Lemma~\ref{lem:small-energy-regularity-sphere} and define
\[
S:=\{x\in\mathbb S^n:\mu(\{x\})\geq\varepsilon_0\}.
\]Since each point of $S$ carries mass at least
$\varepsilon_0$ and $\mu(\mathbb S^n)\leq\Lambda$,
$S$ contains at most $\Lambda/\varepsilon_0$ points.
The following proposition gives compactness in the absence
of concentration.
\begin{prop}\label{S converhence}
For every open set $\Omega$ with
$\overline{\Omega}\subset\mathbb S^n\setminus S$
and every $1<q<\infty$, there exists
$C_{\Omega,q}>0$ such that
\[
\sup_i\|\psi_i\|_{W^{1,q}(\Omega)}
\leq C_{\Omega,q}.
\]
Consequently, after passing to a subsequence,
\[
\psi_i\longrightarrow\psi_\infty
\quad\text{in }
C_{\mathrm{loc}}^{0,\gamma}(\mathbb S^n\setminus S)
\quad\text{for every }0<\gamma<1.
\]
In particular, if $S=\varnothing$, then
\[
\sup_i\|\psi_i\|_{W^{1,q}(\mathbb S^n)}
\leq C_q,
\qquad
\psi_i\longrightarrow\psi_\infty
\quad\text{in }C^{0,\gamma}(\mathbb S^n),
\]
for every $1<q<\infty$ and $0<\gamma<1$.
\end{prop}
\begin{proof}
For each $x\in\mathbb S^n\setminus S$, choose
$0<\rho_x<\rho_0$ sufficiently small that
\[
\mu(B_{\rho_x}^{g_0}(x))<\varepsilon_0,
\qquad
\mu(\partial B_{\rho_x}^{g_0}(x))=0.
\]
Since $B_{\rho_x}^{g_0}(x)$ is a $\mu$-continuity set, 
weak convergence of measures gives
\[
\int_{B_{\rho_x}^{g_0}(x)}
|\psi_i|^{2^*}\,dv_{g_0}
=
\mu_i(B_{\rho_x}^{g_0}(x))
\longrightarrow
\mu(B_{\rho_x}^{g_0}(x))
<\varepsilon_0.
\]
Lemma~\ref{lem:small-energy-regularity-sphere}  therefore gives
\[
\|\psi_i\|_{W^{1,q}(B_{\rho_x/2}^{g_0}(x))}
\leq C_{x,q},
\qquad 1<q<\infty,
\]
for all sufficiently large $i$.

A finite covering of $\overline\Omega$ by these smaller
balls yields the claimed bounds for sufficiently large $i$.
The finitely many remaining terms are absorbed into
$C_{\Omega,q}$ by the regularity of each fixed solution.

For each $0<\gamma<1$, choose $q>n/(1-\gamma)$.
On every smooth domain $\Omega$ with
$\overline{\Omega}\subset\mathbb S^n\setminus S$,
the embedding
\(
W^{1,q}(\Omega)
{\hookrightarrow}
C^{0,\gamma}(\overline{\Omega})
\)
is compact.
The local $W^{1,q}$ bounds and a diagonal argument
therefore yield, after passing to a subsequence,
\[
\psi_i\longrightarrow\psi_\infty
\quad\text{in }
C_{\mathrm{loc}}^{0,\gamma}(\mathbb S^n\setminus S)
\quad\text{for every }0<\gamma<1.
\]
When $S=\varnothing$, the same finite covering argument
and compact embedding give the global bounds
and convergence on $\mathbb S^n$.
\end{proof}

\begin{lem}\label{lem:concentration-blowup-equivalence}
A point $p\in\mathbb S^n$ belongs to $S$ if and only if
it is a blow-up point of $\{\psi_i\}$.
\end{lem}

\begin{proof}
If $p\notin S$, Proposition~\ref{S converhence} gives a uniform
$L^\infty$ bound for $\psi_i$ in a neighborhood of $p$.
Hence $p$ is not a blow-up point.

Conversely, let $p\in S$ and suppose that $p$ is not
a blow-up point.
By a diagonal argument, after passing to a subsequence,
there exist a neighborhood $U$ of $p$ and $C>0$ such that
\[
\sup_U|\psi_i|\leq C.
\]
For every sufficiently small $r>0$ satisfying
$\mu(\partial B_r^{g_0}(p))=0$, weak convergence gives
\[
\mu(B_r^{g_0}(p))
=
\lim_{i\to\infty}
\int_{B_r^{g_0}(p)}|\psi_i|^{2^*}\,dv_{g_0}
\leq
C^{2^*}\operatorname{Vol}_{g_0}(B_r^{g_0}(p)).
\]
Letting $r\to0$ through such radii yields
$\mu(\{p\})=0$, contradicting $p\in S$.
Thus $p$ is a blow-up point.
\end{proof}
\subsection{Blow-up rescaling}\label{Blow-up rescaling near a concentration point}
Assume that $S\neq\varnothing$ and fix $p\in S$.
Since $S$ is finite, choose $\rho>0$ sufficiently small
that
\(
B_{2\rho}^{g_0}(p)\cap S=\{p\}.
\)
Let $\varphi_i$ and $\widetilde H_i$ be obtained by
the conformal transformation in Section~\ref{Conformal covariance and stereographic projection}, and write
\(
U_\rho=\pi_p(B_\rho^{g_0}(p))
=B_{\tan(\rho/2)}(0).
\)
Choose $y_i\in\overline{U_\rho}$ such that
\[
|\varphi_i(y_i)|
=
\max_{\overline{U_\rho}}|\varphi_i|.
\]
By Lemma~\ref{lem:concentration-blowup-equivalence},
$|\varphi_i(y_i)|\to\infty$.
On the other hand, Proposition~\ref{S converhence} gives
\[
\sup_{\overline{U_\rho}\setminus B_\varepsilon(0)}
|\varphi_i|
\leq C_\varepsilon
\]
for every sufficiently small $\varepsilon>0$.
Hence $y_i\to0$, and $y_i\in U_\rho$ for all
sufficiently large $i$.

Set
\[
\lambda_i:=|\varphi_i(y_i)|^{-\frac{2}{n-1}}
\longrightarrow0,
\qquad
\Phi_i(z):=
\lambda_i^{\frac{n-1}{2}}
\varphi_i(y_i+\lambda_i z), 
\]
for $y_i+\lambda_i z\in U_\rho$.
\begin{lem}
\label{lem:rescaled-limit}
Assume that $H_i\to H$ in $C^0(\mathbb S^n)$.
After passing to a subsequence,
\[
\Phi_i\longrightarrow\Phi
\quad\text{in }C_{\mathrm{loc}}^{0,\gamma}(\mathbb R^n)
\quad\text{for every }0<\gamma<1,
\]
where $\Phi$ satisfies
\[
D_{g_{\mathbb R^n}}\Phi
=
H(p)|\Phi|^{2^*-2}\Phi
\quad\text{in }\mathbb R^n
\]
in the distributional sense, and
\[
|\Phi(z)|\leq|\Phi(0)|=1
\quad(z\in\mathbb R^n),
\qquad
\int_{\mathbb R^n}|\Phi|^{2^*}\,dz\leq\Lambda.
\]
\end{lem}

\begin{proof}
Fix $R>0$.
Since $y_i\to0$ and $\lambda_i\to0$, we have
$y_i+\lambda_i B_{2R}\subset U_\rho$
for all sufficiently large $i$.
By the definition of $\Phi_i$,
\[
D_{g_{\mathbb R^n}}\Phi_i
=
\widetilde H_i(y_i+\lambda_i z)
|\Phi_i|^{2^*-2}\Phi_i,
\qquad
|\Phi_i(z)|\leq|\Phi_i(0)|=1
\quad\text{in }B_{2R}.
\]
The uniform bound on $H_i$ and interior elliptic
estimates give
\[
\|\Phi_i\|_{W^{1,q}(B_R)}
\leq C_{R,q},
\qquad 1<q<\infty.
\]
By compact Sobolev embeddings and a diagonal argument,
after passing to a subsequence,
\[
\Phi_i\to\Phi
\quad\text{in }C_{\mathrm{loc}}^{0,\gamma}(\mathbb R^n)
\qquad\text{for every }0<\gamma<1.
\]
In particular, $|\Phi(z)|\le|\Phi(0)|=1$.

Since $H_i\to H$ in $C^0(\mathbb S^n)$,
\(
\widetilde H_i(y_i+\lambda_i z)
\longrightarrow H(p)\) in
\(C^0(\overline{B_R}).
\)
Passing to the limit in the distributional formulation
gives the limiting equation.

Finally, scaling and conformal invariance of the
critical energy give
\[
\begin{aligned}
\int_{B_R}|\Phi_i|^{2^*}\,dz
&=
\int_{y_i+\lambda_i B_R}|\varphi_i|^{2^*}\,dy
\leq
\int_{U_\rho}|\varphi_i|^{2^*}\,dy\\
&=
\int_{B_\rho^{g_0}(p)}
|\psi_i|^{2^*}\,dv_{g_0}
\leq\Lambda.
\end{aligned}
\]
Passing to the limit as $i\to\infty$ and then
letting $R\to\infty$ proves the energy bound.
\end{proof}
\begin{prop}
\label{prop:compactness-alternative}
Let $\{\psi_i\}$ be a sequence
of distributional solutions of \eqref{eq:intro-sequence} and 
satisfying \eqref{eq:intro-assumptions}. After passing to a subsequence, exactly one of the
following alternatives holds:
\begin{enumerate}
\item
$S=\varnothing$, and
\(
\psi_i\longrightarrow\psi_\infty\) in \(
C^{0,\gamma}(\mathbb S^n)
\) for every $0<\gamma<1$.

\item
$S\neq\varnothing$ contains at most
$\frac{\Lambda}{\varepsilon_0}$ points, and
\(
\psi_i\longrightarrow\psi_\infty\) in \(
C_{\mathrm{loc}}^{0,\gamma}(\mathbb S^n\setminus S)
\) for every \(0<\gamma<1\).
Moreover, for each $p\in S$, the rescaled spinors
$\Phi_i$ converge to a nonzero limit $\Phi$
with the properties stated in Lemma~\ref{lem:rescaled-limit}.
\end{enumerate}
\end{prop}
\begin{proof}
Proposition~\ref{S converhence} gives the asserted compactness,
while Lemmas~\ref{lem:concentration-blowup-equivalence} and~\ref{lem:rescaled-limit} give the rescaled limits.
Since $S$ is finite, a common subsequence can be
chosen for all $p\in S$.
\end{proof}
\[
\]
\section{Proof of Theorem~\ref{thm:intro-threshold}}\label{Proof of Theorem 1.1}
\label{sec:first-energy-threshold}
Recall that every nontrivial finite-energy solution of
\begin{equation}\label{eq:standard-critical-dirac}
D_{g_{\R^n}}\varphi
=
|\varphi|^{2^*-2}\varphi
\quad\text{in }\mathbb R^n
\end{equation}
satisfies 
\begin{equation}\label{eq:standard-dirac-energy}
\int_{\mathbb R^n}|\varphi|^{2^*}\,dy
\ge \left(\frac n2\right)^n\omega_n,
\end{equation}
where $\omega_n=\operatorname{Vol}_{g_0}(\mathbb S^n)$;
see \cite{BMW2021}.
A solution attaining equality is called a ground-state solution.
The bound also applies to distributional solutions in
$L^{2^*}(\mathbb R^n,\Sigma\mathbb R^n)$, since the equation
and the Hardy--Littlewood--Sobolev inequality imply that
such solutions belong to
$\mathcal D^{1/2}(\mathbb R^n,\Sigma\mathbb R^n)$.
\begin{prop}
\label{prop:first-energy-threshold}
Assume the hypotheses of
Proposition~\ref{prop:compactness-alternative},
with $H_i\ge h_0>0$ .
Suppose that
\begin{equation}\label{eq:first-energy-bound}
\limsup_{i\to\infty}\int_{\mathbb S^n}|\psi_i|^{2^*}\,dv_{g_0}
\le H_{\max}^{-n}\left(\frac n2\right)^n\omega_n,
\end{equation} where $H_{\max}=\max_{\mathbb S^n}H$. If $S\ne\varnothing$, then $S=\{p\}$,
$H(p)=H_{\max}$, and
\[
|\psi_i|^{2^*}\,dv_{g_0}
\rightharpoonup
H_{\max}^{-n}\left(\frac n2\right)^n\omega_n\,\delta_p
\qquad\text{weakly as measures on }\mathbb S^n.
\]
Moreover, let $\Phi$ be any limit obtained in Lemma~\ref{lem:rescaled-limit}
at $p$.
Then $H(p)^{\frac{n-1}{2}}\Phi$ is a ground-state solution
of \eqref{eq:standard-critical-dirac}.
\end{prop}
\begin{proof}
  Fix $p\in S$ and let $\Phi$ be any limit obtained in
Lemma~\ref{lem:rescaled-limit} at $p$. Then $\Phi$ is nontrivial and satisfies
\[
D_{g_{\R^n}}\Phi
=
H(p)|\Phi|^{2^*-2}\Phi
\qquad\text{in }\mathbb R^n.
\]
Since $H(p)>0$, applying \eqref{eq:standard-dirac-energy}
to $H(p)^{\frac{n-1}{2}}\Phi$ gives
\[
\int_{\mathbb R^n}|\Phi|^{2^*}\,dy
\ge
H(p)^{-n}\left(\frac n2\right)^n\omega_n.
\]

On the other hand, local convergence of the rescaled spinors
and conformal invariance of the critical integral imply that,
for every fixed $R>0$,
\[
\int_{B_R}|\Phi|^{2^*}\,dy
\le
\liminf_{i\to\infty}
\int_{\mathbb S^n}|\psi_i|^{2^*}\,dv_{g_0}.
\]
Letting $R\to\infty$, using \eqref{eq:first-energy-bound}
and $H(p)\le H_{\max}$, we obtain
\[
\begin{aligned}
H(p)^{-n}\left(\frac n2\right)^n\omega_n
&\le \int_{\mathbb R^n}|\Phi|^{2^*}\,dy\le \liminf_{i\to\infty}
     \int_{\mathbb S^n}|\psi_i|^{2^*}\,dv_{g_0}\le H_{\max}^{-n}\left(\frac n2\right)^n\omega_n\le H(p)^{-n}\left(\frac n2\right)^n\omega_n.
\end{aligned}
\]
Thus all inequalities are equalities. In particular,
$H(p)=H_{\max}$ and
$H(p)^{\frac{n-1}{2}}\Phi$ is a ground state.

Let $U$ be any open neighborhood of $p$.
For every fixed $R>0$, we have
\[
\pi_p^{-1}\bigl(B_{R\lambda_i}(y_i)\bigr)\subset U
\]
for all sufficiently large $i$.
By the energy bound, conformal invariance, and the
local convergence of $\Phi_i$, we obtain
\[
\limsup_{i\to\infty}
\int_{\mathbb S^n\setminus U}|\psi_i|^{2^*}\,dv_{g_0}
\le
H_{\max}^{-n}\left(\frac n2\right)^n\omega_n
-\int_{B_R}|\Phi|^{2^*}\,dz.
\]
Letting $R\to\infty$, the right-hand side tends to zero.
Since the preceding equalities also give
\[
\lim_{i\to\infty}
\int_{\mathbb S^n}|\psi_i|^{2^*}\,dv_{g_0}
=
H_{\max}^{-n}\left(\frac n2\right)^n\omega_n,
\]
the asserted measure convergence follows.
In particular, $S=\{p\}$.
\end{proof}
\begin{cor}
\label{cor:first-bubble}
Under the assumptions of Proposition~\ref{prop:first-energy-threshold}, let $\Phi$ be any limit obtained in
Lemma~\ref{lem:rescaled-limit} at $p$.
Then $\Psi:=H(p)^{\frac{n-1}{2}}\Phi$ admits the representation
\[
\Psi(y)
=
\sqrt{\lambda}
\left(\frac{2\lambda}{\lambda^2+|y-x_0|^2}\right)^{\frac n2}
\left(1-
\left(\frac{y-x_0}{\lambda}\right)\right)\cdot_{g_{\Rn}}
\widetilde{\Phi}_0,
\]
where
\[
x_0=0,\qquad
\lambda=\frac{n}{H(p)},\qquad
|\widetilde{\Phi}_0|
=\frac{1}{\sqrt{2}}\left(\frac n2\right)^{\frac{n-1}{2}},
\]
and $\widetilde{\Phi}_0\in\mathbb C^{2^{[ \frac{n}{2}]}}$ is a constant spinor.
In particular,
\begin{equation}\label{normalized-bubble-norm}
|\Phi(y)|
=
\left(1+\frac{H(p)^2}{n^2}|y|^2\right)^{-\frac{n-1}{2}}.
\end{equation}
\end{cor}
\begin{proof}
By Proposition~\ref{prop:first-energy-threshold} and
\cite[Corollary~1.4]{BMW2021}, there exist
\(x_0\in\mathbb R^n\), \(\lambda>0\), and a constant spinor
\(\widetilde\Phi_0\in\mathbb C^N\), with
\[
|\widetilde\Phi_0|
=
\frac1{\sqrt2}\left(\frac n2\right)^{\frac{n-1}{2}},
\]
such that
\[
\Psi(y)
=\sqrt{\lambda}
\left(
\frac{2\lambda}{\lambda^2+|y-x_0|^2}
\right)^{\frac n2}
\left(
1-\left(\frac{y-x_0}{\lambda}\right)
\right)\cdot_{g_{\Rn}}
\widetilde\Phi_0 .
\]
Taking norms and using the Clifford relations gives
\[
|\Phi(y)|
=
\left(
\frac{n\lambda}{H(p)(\lambda^2+|y-x_0|^2)}
\right)^{\frac{n-1}{2}}.
\]
The normalization in Lemma~\ref{lem:rescaled-limit}
therefore implies $x_0=0$ and
\(
1=|\Phi(0)|
=(\frac{n}{H(p)\lambda})^{\frac{n-1}{2}}.
\)
Thus $\lambda=\frac{n}{H(p)}$.
\end{proof}
We now combine the preceding results to prove
Theorem~\ref{thm:intro-threshold}.
\begin{proof}[\normalfont\bfseries Proof of Theorem~\ref{thm:intro-threshold}]
Since $\|\psi_i\|_{L^\infty(\mathbb S^n)}\to\infty$,
Proposition~\ref{prop:compactness-alternative} gives $S\ne\varnothing$ after passing to
a subsequence. Proposition~\ref{prop:first-energy-threshold} then proves~(i).

Choose $y_i$, $\lambda_i$, and $\Phi_i$ as in Section~\ref{Blow-up rescaling near a concentration point},
with $\Phi_i$ defined on $\mathbb R^n$ by the same
rescaling formula. By Lemma~\ref{lem:rescaled-limit} and Corollary~\ref{cor:first-bubble},
\(
\Phi_i\longrightarrow\Phi\) in \(
C^0_{\mathrm{loc}}(\mathbb R^n),
\)
where $\Phi$ has the form stated in~(ii).
It remains to prove that
\(
\Phi_i\longrightarrow\Phi\) in \(
L^{2^*}(\mathbb R^n).
\)
By conformal invariance and Proposition~\ref{prop:first-energy-threshold},
\begin{equation}\label{total energy exhausts}
\int_{\mathbb R^n}|\Phi_i|^{2^*}\,dz
=\int_{\mathbb S^n}|\psi_i|^{2^*}\,dv_{g_0}\longrightarrow
H_{\max}^{-n}\left(\frac n2\right)^n\omega_n
=\int_{\mathbb R^n}|\Phi|^{2^*}\,dz.
\end{equation}
Together with local uniform convergence, this implies,
for every $R>0$,
\[
\limsup_{i\to\infty}
\int_{\mathbb R^n}|\Phi_i-\Phi|^{2^*}\,dz
\le
2^{2^*}\int_{\mathbb R^n\setminus B_R}
|\Phi|^{2^*}\,dz.
\]
Letting $R\to\infty$ proves convergence in
$L^{2^*}(\mathbb R^n)$ and completes~(ii).
\end{proof}

\section{Proof of Theorem~\ref{thm:intro-simple} }\label{Proof of Theorem 1.2}
Throughout this section, we assume the hypotheses of
Theorem~\ref{thm:intro-threshold} and
$\|\psi_i\|_{L^\infty(\mathbb S^n)}\to\infty$.
Let $p$ be the unique blow-up point, and retain the
spinors $\varphi_i$ and the rescalings $\Phi_i$
introduced in Subsection~\ref{Blow-up rescaling near a concentration point}.
\subsection{Isolated blow-up points}
\begin{prop}\label{prop:isolation}The point $p$ is an isolated blow-up point.
\end{prop}
\begin{proof}
  We retain the local maximum points $y_i$ chosen in Subsection~\ref{Blow-up rescaling near a concentration point}.
For $\rho>0$ sufficiently small, $U_\rho\subset B_\rho(0)$.
By the local bounds away from the unique concentration
point, $\varphi_i$ is uniformly bounded on
$\overline{B_\rho(0)}\setminus U_\rho$.
Since $|\varphi_i(y_i)|\to\infty$, it follows that,
for all sufficiently large $i$,
\[
y_i\in B_\rho(0),
\qquad
|\varphi_i(y_i)|
=
\max_{\overline{B_\rho(0)}}|\varphi_i|.
\]
By Theorem~\ref{thm:intro-threshold}, after passing to
a subsequence,
\(
\Phi_i\longrightarrow\Phi\) in \(L^{2^*}(\mathbb R^n).
\)
Let $\varepsilon_0$ be the small-energy constant in
Corollary~\ref{cor:epsilon-regularity-scale}. Choose $R>0$ sufficiently large that
\[
\int_{\mathbb R^n\setminus B_R}
|\Phi|^{2^*}\,dz<\frac{\varepsilon_0}{2}.
\]
By the strong convergence and a change of variables,
for all sufficiently large $i$,
\[
\int_{B_\rho(0)\setminus B_{R\lambda_i}(y_i)}
|\varphi_i|^{2^*}\,dy
\le
\int_{\mathbb R^n\setminus B_R}
|\Phi_i|^{2^*}\,dz<\varepsilon_0.
\]
We claim that there exists \(C>0\), independent of \(i\), such that
\[
|\varphi_i(y)|
\leq
C |y-y_i|^{-\frac{n-1}{2}},
\qquad
0<|y-y_i|<\frac{\rho}{2}.
\]
Since \(R\) is fixed and \(\lambda_i\to0\), while \(y_i\to0\), we may assume, for all sufficiently large \(i\), that
\[
4R\lambda_i<\frac{\rho}{4},
\qquad
|y_i|<\frac{\rho}{8}.
\]
If $ |y-y_i|\leq 4R\lambda_i$, then the maximality
of $|\varphi_i(y_i)|$ in $B_\rho(0)$ gives
\[
|\varphi_i(y)|
\leq
|\varphi_i(y_i)|
=
\lambda_i^{-\frac{n-1}{2}}
\leq
(4R)^{\frac{n-1}{2}}
|y-y_i|^{-\frac{n-1}{2}}\leq
C|y-y_i|^{-\frac{n-1}{2}} .
\]If
$|y-y_i|>4R\lambda_i$, 
then
\(
B_{\frac{|y-y_i|}{4}}(y)
\subset
B_\rho(0)\setminus B_{R\lambda_i}(y_i)
\)
for all sufficiently large \(i\). Therefore
\[
\int_{B_{\frac{|y-y_i|}{4}}(y)}|\varphi_i|^{2^*}\,dy
\leq
\int_{B_\rho(0)\setminus B_{R\lambda_i}(y_i)}
|\varphi_i|^{2^*}\,dy
<
\varepsilon_0.
\]
By Corollary~\ref{cor:epsilon-regularity-scale} and the preceding energy bound,
\[
\begin{aligned}
|\varphi_i(y)|
&\le
C\left(\frac{|y-y_i|}{8}\right)^{-\frac{n-1}{2}}
\left(
\int_{B_{\frac{|y-y_i|}{4}}(y)}
|\varphi_i(z)|^{2^*}\,dz
\right)^{1/2^*}\\
&\le C|y-y_i|^{-\frac{n-1}{2}}.
\end{aligned}
\]
Combining the two cases proves the claim.
Hence $p$ is an isolated blow-up point.
\end{proof}
\subsection{Kelvin compactness}
\begin{lem}
\label{lem:kelvin-compactness} Let $p$, $y_i$, $\lambda_i$, $\Phi_i$, and $\Phi$ be
as in Subsection~\ref{Blow-up rescaling near a concentration point}, with the subsequence fixed in Lemma~\ref{lem:rescaled-limit}.
Using the Euclidean representatives $\varphi_i$ on all of
$\mathbb R^n$, define the same rescalings globally by
\[
 \Phi_i(z):=\lambda_i^{\frac{n-1}{2}}
             \varphi_i(y_i+\lambda_i z),
 \qquad z\in\mathbb R^n.
\]
For $x\ne0$, set
\[
 \Phi_i^K(x):=\frac{x}{|x|^n}\cdot
                 \Phi_i\!\left(\frac{x}{|x|^2}\right),
 \qquad
 \Phi^K(x):=\frac{x}{|x|^n}\cdot
                 \Phi\!\left(\frac{x}{|x|^2}\right).
\]
Then for all sufficiently large $i$, 
\begin{equation}\label{eq:kelvin-convergence}
 \Phi_i^K\longrightarrow\Phi^K
 \quad\text{in }C^0_{\mathrm{loc}}(\mathbb R^n),
 \qquad \Phi^K(0)=\left(\frac n{H(p)}\right)^{n-1}\Phi(0)\ne0.
\end{equation}
\begin{equation}\label{eq:kelvin-tail}
 |\Phi_i(z)|\le C(1+|z|)^{1-n}
 \qquad(z\in\mathbb R^n),
\end{equation}
where $C$ is independent of $i$.

Consequently, for every sequence $\mu_i>0$ with
$\mu_i/\lambda_i\to\infty$,
\begin{equation}\label{eq:kelvin-second-scale}
 \left(\frac{\mu_i}{\lambda_i}\right)^{n-1}
 \Phi_i\!\left(\frac{\mu_i}{\lambda_i}z\right)
 \longrightarrow
 -\left(\frac n{H(p)}\right)^{n-1}\frac{z}{|z|^n}\cdot\Phi(0)
 \quad\text{in }C^0_{\mathrm{loc}}
       (\mathbb R^n\setminus\{0\}).
\end{equation}
\end{lem}

\begin{proof}By \eqref{total energy exhausts}  and the local convergence in Lemma~\ref{lem:rescaled-limit},
for every $R>0$,
\[
\int_{|z|>R}|\Phi_i|^{2^*}\,dz
=
\int_{\mathbb S^n}|\psi_i|^{2^*}\,dv_{g_0}
-\int_{B_R}|\Phi_i|^{2^*}\,dz\longrightarrow
\int_{|z|>R}|\Phi|^{2^*}\,dz.
\]
Consequently,
\begin{equation}\label{eq:kelvin-energy-tail 1}
\lim_{R\to\infty}\limsup_{i\to\infty}
\int_{|z|>R}|\Phi_i|^{2^*}\,dz=0.
\end{equation}

Clifford multiplication and inversion give
\begin{equation}\label{eq:kelvin-energy-tail 2}
 |\Phi_i^K(x)|
 =|x|^{1-n}\left|\Phi_i\!\left(\frac{x}{|x|^2}\right)\right|,
 \qquad
 \int_{B_r}|\Phi_i^K|^{2^*}\,\mathrm dx
 =\int_{|z|>1/r}|\Phi_i|^{2^*}\,\mathrm dz.
\end{equation}
Direct differentiation gives
\begin{equation}\label{eq:kelvin-equation}
  \begin{aligned}
D_{g_{\mathbb R^n}}\Phi_i^K(x)
&=
-\frac{x}{|x|^{n+2}}\cdot
(D_{g_{\mathbb R^n}}\Phi_i)\left(\frac{x}{|x|^2}\right)\\
&=-\widetilde H_i\!\left(y_i+\lambda_i\frac{x}{|x|^2}\right)
 |\Phi_i^K|^{2^*-2}\Phi_i^K
 \quad\text{in }\mathbb R^n\setminus\{0\}.
\end{aligned}
\end{equation}
The above equation also holds on all of $\mathbb R^n$
in the distributional sense. Let $\zeta\in C_c^\infty(\mathbb R^n,\Sigma\mathbb R^n)$,
and choose $0\le\chi_\varepsilon\le1$ such that
$\chi_\varepsilon=0$ on $B_\varepsilon$,
$\chi_\varepsilon=1$ outside $B_{2\varepsilon}$, and
$|\nabla\chi_\varepsilon|\le C\varepsilon^{-1}$.
Testing with $\chi_\varepsilon\zeta$, the additional term
satisfies
\[
\begin{aligned}
\left|
\int_{\mathbb R^n}
\langle\Phi_i^K,\nabla\chi_\varepsilon\cdot\zeta\rangle\,dx
\right|
&\le
C\|\zeta\|_\infty\varepsilon^{-1}
\int_{B_{2\varepsilon}}|\Phi_i^K|\,dx\\
&\le
C\|\zeta\|_\infty\varepsilon^{\frac{n-1}{2}}
\|\Phi_i^K\|_{L^{2^*}(B_{2\varepsilon})}
\longrightarrow0.
\end{aligned}
\]
Since $\Phi_i^K\in L^{2^*}$ and the coefficient is bounded,
the remaining terms converge by dominated convergence.
Thus the equation holds distributionally on $\mathbb R^n$.
By \eqref{eq:kelvin-energy-tail 1} and \eqref{eq:kelvin-energy-tail 2}, choose a fixed $\rho>0$ so small that,
for all sufficiently large $i$,
\[
 \int_{B_{8\rho}}|\Phi_i^K|^{2^*}<\varepsilon_0,
\]where $\varepsilon_0$ is the constant in Corollary~\ref{cor:epsilon-regularity-scale}.
By Corollary~~\ref{cor:epsilon-regularity-scale} and interior elliptic estimates, we obtain
\[
\|\Phi_i^K\|_{L^\infty(B_{4\rho})}
+
\|\Phi_i^K\|_{W^{1,2n}(B_{2\rho})}
\le C.
\]
Hence $\{\Phi_i^K\}$ is uniformly bounded in
$C^{0,1/2}(\overline{B_\rho})$ and is precompact in
$C^0(\overline{B_\rho})$.

By Corollary~\ref{cor:first-bubble},
\[
\Phi^K(x)
=\left(|x|^2+\frac{H(p)^2}{n^2}\right)^{-n/2}
\left(\frac{H(p)}n+x\cdot\right)\Phi(0).
\]
Thus $\Phi^K$ extends continuously at $0$, with
\[
\Phi^K(0)
=\left(\frac n{H(p)}\right)^{n-1}\Phi(0)\ne0.
\]By Lemma~\ref{lem:rescaled-limit} and the definition of the Kelvin transform,
$\Phi_i^K\to\Phi^K$ locally uniformly away from $0$.
Together with the above compactness and the continuity
of $\Phi^K$ at $0$, this yields
\[
\Phi_i^K\longrightarrow\Phi^K
\quad\text{in }C^0_{\mathrm{loc}}(\mathbb R^n).
\]
For $|z|\ge\rho^{-1}$, the uniform bound on $B_\rho$ gives
\[
|\Phi_i(z)|
=|z|^{1-n}
 \left|\Phi_i^K\left(\frac{z}{|z|^2}\right)\right|
\le C|z|^{1-n}.
\]
Together with the local uniform boundedness of $\Phi_i$,
this yields
\[
|\Phi_i(z)|\le C(1+|z|)^{1-n},
\qquad z\in\mathbb R^n.
\]

Finally, if $\mu_i/\lambda_i\to\infty$, then
\[
\begin{aligned}
\left(\frac{\mu_i}{\lambda_i}\right)^{n-1}
\Phi_i\left(\frac{\mu_i}{\lambda_i}z\right)
&=-\frac{z}{|z|^n}\cdot
\Phi_i^K\left(
\frac{\lambda_i}{\mu_i}\frac{z}{|z|^2}
\right)\\
&\longrightarrow
-\frac{z}{|z|^n}\cdot\Phi^K(0)  \quad\text{in } C^0_{\mathrm{loc}}(\mathbb R^n\setminus\{0\}).
\end{aligned}
\]

\end{proof}
\subsection{Proof of isolated simplicity}

We now work under the hypotheses of Theorem~\ref{thm:intro-simple}.
Let $U$ be a neighborhood of $p$ such that
\[
\sup_i\|H_i\|_{C^1(U)}<\infty.
\]
Choose $r_0>0$ such that
$\pi_p^{-1}(\overline{B_{2r_0}})\subset U$. Then
\[
\sup_i\|\widetilde H_i\|_{C^1(B_{2r_0})}<\infty.
\]
By Propositions~\ref{S converhence} and~\ref{prop:isolation}, there exists $A_3>0$,
independent of $i$, such that, for all sufficiently large $i$,
\begin{equation}\label{isolation bound}
|y-y_i|^{\frac{n-1}{2}}|\varphi_i(y)|\le A_3,
\qquad 0<|y-y_i|<1.
\end{equation}
\begin{lem}[Second-scale reduction]\label{second scale reduction}Suppose that $0$ is an isolated blow-up point satisfying
\eqref{isolation bound}, but is not simple. Then, after passing to a subsequence,
there exist second critical points $\mu_i$ of $\bar w_i$,
with $\mu_i\to0$ and $\mu_i/\lambda_i\to\infty$, such that 
\(
\xi_i(z):=\mu_i^{\frac{n-1}{2}}
\varphi_i(y_i+\mu_i z)
\)
satisfy
\[
D_{g_{\mathbb R^n}}\xi_i
=\widehat H_i|\xi_i|^{2^*-2}\xi_i,
\qquad
\widehat H_i(z):=\widetilde H_i(y_i+\mu_i z),
\]
\[
|\xi_i(0)|\to\infty,
\qquad
|\xi_i(z)|\le C|z|^{-\frac{n-1}{2}}
\quad
\left(0<|z|<\frac{1}{\mu_i}\right).
\]
Moreover, $0$ is an isolated simple blow-up point for
$\{\xi_i\}$, and
\[
\left.
\frac{d}{dr}
\left(
\frac{r^{n-1}}{|\partial B_r|}
\int_{\partial B_r}|\xi_i|^2\,d\sigma
\right)
\right|_{r=1}=0.
\]
\end{lem}

\begin{proof}
  Fix $R>0$. For all sufficiently large $i$,
\(
y_i+\lambda_i B_{4R}\subset B_{2r_0}\cap U_\rho.
\)
By the choice of $y_i$, we have
\[
D_{g_{\mathbb R^n}}\Phi_i
=
\widetilde H_i(y_i+\lambda_i z)
|\Phi_i|^{2^*-2}\Phi_i,
\qquad
|\Phi_i|\le1
\quad\text{in }B_{4R}.
\]
Moreover,
\[
\bigl\|\widetilde H_i(y_i+\lambda_i\,\cdot)\bigr\|_{L^\infty(B_{4R})}
\le C,
\qquad
\bigl\|\nabla_z
[\widetilde H_i(y_i+\lambda_i\,\cdot)]\bigr\|_{L^\infty(B_{4R})}
\le C\lambda_i.
\]
Since the map $z\mapsto |z|^{2^*-2}z$ is continuously
differentiable in the real sense, interior elliptic estimates
and the chain rule give
\[
\|\Phi_i\|_{W^{2,2n}(B_{2R})}
+
\|\Phi_i\|_{C^{1,1/2}(\overline{B_{2R}})}
\le C_R.
\]
Since
\[
\Delta|\Phi_i|^2
=
2|\nabla\Phi_i|^2
-
2\widetilde H_i(y_i+\lambda_i z)^2
|\Phi_i|^{2(2^*-1)},
\]
the right-hand side is locally uniformly bounded
in $C^{0,1/2}$.
Scalar interior Schauder estimates and compactness
therefore give
\[
|\Phi_i|^2\longrightarrow|\Phi|^2
\quad\text{in }C^2_{\mathrm{loc}}(\mathbb R^n).
\]
Using Corollary~\ref{cor:first-bubble}, we obtain
\[
\overline w_i(\lambda_i s)=
\frac{s^{n-1}}{|\mathbb S^{n-1}|}
\int_{\mathbb S^{n-1}}|\Phi_i(s\theta)|^2\,d\theta
\longrightarrow
\left(
\frac{s}{1+\frac{H(p)^2}{n^2}s^2}
\right)^{n-1}
\quad\text{in }C^2_{\mathrm{loc}}((0,\infty)).
\]
The limit has a unique nondegenerate critical point at
$s=n/H(p)$.
Moreover, $|\Phi_i(0)|=1$ and the uniform $C^1$ bounds
imply that $\overline w_i(\lambda_i s)$ is strictly increasing
for all sufficiently small $s>0$, uniformly in $i$.
Consequently, for every fixed $R>n/H(p)$,
$\overline w_i$ has precisely one critical point in
$(0,R\lambda_i)$ for all sufficiently large $i$.

Since $0$ is not simple, for every $\varepsilon\in(0,1)$
there are arbitrarily large $i$ for which $\bar w_i$ has
at least two critical points in $(0,\varepsilon)$.
Thus, after passing to a subsequence, the second critical
point $\mu_i$ satisfies $0<\mu_i<1/i$.
For every fixed $R>n/H(p)$, the preceding uniqueness implies
$\mu_i\ge R\lambda_i$ for all sufficiently large $i$.
Hence $\mu_i/\lambda_i\to\infty$.

Set
\[
\xi_i(z):=\mu_i^{\frac{n-1}{2}}
\varphi_i(y_i+\mu_i z),
\qquad |z|<\mu_i^{-1}.
\]By the Euclidean equation and 
\eqref{isolation bound}, the rescaled spinors satisfy
\[
\begin{aligned}
D_{g_{\mathbb R^n}}\xi_i
&=\widehat H_i|\xi_i|^{2^*-2}\xi_i,
&& |z|<\mu_i^{-1},\\
|z|^{\frac{n-1}{2}}|\xi_i(z)|
&\le A_3,
&& 0<|z|<\mu_i^{-1},\\
|\xi_i(0)|
&=\left(\frac{\mu_i}{\lambda_i}\right)^{\frac{n-1}{2}}
\longrightarrow\infty,
\end{aligned}
\]where $\widehat H_i(z):=\widetilde H_i(y_i+\mu_i z)$.
Moreover, by the definition of $\xi_i$,
\[
\frac{r^{n-1}}{|\partial B_r|}
\int_{\partial B_r}|\xi_i|^2\,d\sigma
=
\bar w_i(\mu_i r).
\]
Since $\mu_i$ is the second critical point of $\bar w_i$,
the left-hand side has exactly one critical point in $(0,1)$,
and
\[
\left.
\frac{d}{dr}
\left(
\frac{r^{n-1}}{|\partial B_r|}
\int_{\partial B_r}|\xi_i|^2\,d\sigma
\right)
\right|_{r=1}
=
\mu_i\bar w_i'(\mu_i)=0.
\]
Thus, $0$ is an isolated simple blow-up point for $\{\xi_i\}$.
\end{proof}

\begin{lem}
\label{lem:simp-squared-norm}
Let $\Omega\subset\mathbb R^n$ be open. Suppose that $u_i\in L^\infty_{\mathrm{loc}}(\Omega,\Sigma\mathbb R^n)$
satisfy
\[
 D_{g_{\mathbb R^n}}u_i=V_i u_i\quad\text{in }\mathcal D'(\Omega,\Sigma\mathbb R^n),
\]
where $V_i$ are real-valued and, for every $U\Subset\Omega$,
\[
 \sup_i\bigl(\|u_i\|_{L^\infty(U)}
              +\|V_i\|_{L^\infty(U)}\bigr)<\infty.
\]Then
\(
\sup_i\bigl\||u_i|^2\bigr\|_{W^{2,p}(U)}<\infty\) for every \( 1<p<\infty.
\) In particular,
if $u_i\to u$ in $C^0_{\mathrm{loc}}(\Omega)$, then
\[
|u_i|^2\to |u|^2
\quad\text{in }C^1_{\mathrm{loc}}(\Omega).
\]
\end{lem}

\begin{proof}
For $1<p<\infty$ and $U\Subset U'\Subset\Omega$,
interior elliptic estimates give
\[
\begin{aligned}
\|u_i\|_{W^{1,p}(U)}
&\le C\bigl(
\|D_{g_{\mathbb R^n}}u_i\|_{L^p(U')}
+\|u_i\|_{L^p(U')}\bigr)\\
&\le C\bigl(1+\|V_i\|_{L^\infty(U')}\bigr)
\|u_i\|_{L^p(U')}
\le C.
\end{aligned}
\]
For any $\chi\in C_c^\infty(\Omega,\mathbb R)$,
integration by parts and
$D_{g_{\mathbb R^n}}^2=-\Delta$, with
$\Delta=\sum_{j=1}^n\partial_j^2$, give, by density,
\[
\begin{aligned}
\frac12\int_\Omega |u_i|^2\Delta\chi
&=\int_\Omega\chi|\nabla u_i|^2
-\int_\Omega
( D_{g_{\mathbb R^n}}u_i,
D_{g_{\mathbb{R}^n}}(\chi u_i))\\
&=\int_\Omega\chi
\bigl(|\nabla u_i|^2-V_i^2|u_i|^2\bigr).
\end{aligned}
\]
Thus
\begin{equation}\label{eq:simp-norm-identity}
\Delta|u_i|^2
=2|\nabla u_i|^2-2V_i^2|u_i|^2
\quad\text{in }\mathcal D'(\Omega).
\end{equation}Combining \eqref{eq:simp-norm-identity} with interior elliptic estimates,
we obtain
\[
\begin{aligned}
\bigl\||u_i|^2\bigr\|_{W^{2,p}(U)}
&\le C\left(
\bigl\|\Delta|u_i|^2\bigr\|_{L^p(U')}
+\bigl\||u_i|^2\bigr\|_{L^p(U')}\right)\\
&\le C\left(
\|\nabla u_i\|_{L^{2p}(U')}^2
+\bigl(1+\|V_i\|_{L^\infty(U')}^2\bigr)
\|u_i\|_{L^{2p}(U')}^2\right)
\le C.
\end{aligned}
\]
If $u_i\to u$ in $C^0_{\mathrm{loc}}(\Omega)$,
the preceding bounds with $p>n$ and compact Sobolev embedding
yield 
\[
|u_i|^2\to |u|^2
\quad\text{in }C^1_{\mathrm{loc}}(\Omega).
\]
\end{proof}

\begin{lem}[The normalized second-scale limit]
\label{lem:simp-second-scale-limit}
Let $\mu_i$ and $\xi_i$ be as in Lemma~\ref{second scale reduction}.
With the notation of Lemma~\ref{lem:kelvin-compactness} and along the subsequence
fixed there, we have
\begin{equation}\label{eq:simp-second-limit}
 |\xi_i(0)|\xi_i (z)\longrightarrow-\left(\frac n{H(p)}\right)^{n-1}\frac{z}{|z|^n}\cdot\Phi(0)
 \quad\text{in }C^0_{\mathrm{loc}}(\mathbb R^n\setminus\{0\}).
\end{equation}
Moreover,
\begin{equation}\label{eq:simp-second-norm-limit}
 |\xi_i(0)|^2|\xi_i(z)|^2\longrightarrow
\left(\frac{n}{H(p)}\right)^{2n-2}|z|^{2-2n}
 \quad\text{in }C^1_{\mathrm{loc}}(\mathbb R^n\setminus\{0\}).
\end{equation}
\end{lem}
\begin{proof}
By definition,
\[
|\xi_i(0)|
=\left(\frac{\mu_i}{\lambda_i}\right)^{\frac{n-1}{2}},
\qquad|\xi_i(0)|\xi_i(z)
=\left(\frac{\mu_i}{\lambda_i}\right)^{n-1}
\Phi_i\left(\frac{\mu_i}{\lambda_i}z\right).
\]Thus, \eqref{eq:simp-second-limit} follows from Lemma~\ref{lem:kelvin-compactness}.
Moreover,
\[
D_{g_{\mathbb R^n}}\bigl(|\xi_i(0)|\xi_i\bigr)
=\widehat H_i|\xi_i|^{2^*-2}
\bigl(|\xi_i(0)|\xi_i\bigr).
\]
Fix a bounded open set $U$ in $\mathbb R^n$ with
$0\notin\overline U$. By \eqref{eq:simp-second-limit},
\(
\bigl\||\xi_i(0)|\xi_i\bigr\|_{L^\infty(U)}
\le C_U.
\)
Since $|\xi_i(0)|\to\infty$ and $\widehat H_i$ are uniformly bounded,
we obtain
\[
\begin{aligned}
\bigl\|\widehat H_i|\xi_i|^{2^*-2}\bigr\|_{L^\infty(U)}
&\le
C|\xi_i(0)|^{-(2^*-2)}
\bigl\||\xi_i(0)|\xi_i\bigr\|_{L^\infty(U)}^{2^*-2}\\
&\le C|\xi_i(0)|^{-(2^*-2)}
\longrightarrow0.
\end{aligned}
\]
Thus, the coefficient is real-valued and locally uniformly
bounded on $\mathbb R^n\setminus\{0\}$. Applying Lemma~\ref{lem:simp-squared-norm}
to $|\xi_i(0)|\xi_i$ and using \eqref{eq:simp-second-limit}, we conclude that
\[
|\xi_i(0)|^2|\xi_i(z)|^2
\longrightarrow
\left|
-\left(\frac{n}{H(p)}\right)^{n-1}
\frac{z}{|z|^n}\cdot\Phi(0)
\right|^2
\quad\text{in }
C^1_{\mathrm{loc}}(\mathbb R^n\setminus\{0\}).
\]
Since $|\Phi(0)|=1$ and Clifford multiplication satisfies
$|z\cdot\Phi(0)|^2=|z|^2|\Phi(0)|^2$, the limit equals
\(
\left(\frac{n}{H(p)}\right)^{2n-2}|z|^{2-2n}.
\)
This proves \eqref{eq:simp-second-norm-limit}.
\end{proof}
\begin{proof}[\normalfont\bfseries Proof of Theorem~\ref{thm:intro-simple}]
By Proposition~\ref{prop:isolation}, $p$ is an isolated blow-up point.
Suppose, for a contradiction, that $p$ is not simple. Let $\mu_i$ and $\xi_i$ be given by Lemma~~\ref{second scale reduction}. 
By Lemma~\ref{lem:simp-second-scale-limit},
\[
\frac{r^{n-1}}{|\mathbb S^{n-1}|}
\int_{\mathbb S^{n-1}}
|\xi_i(0)|^2|\xi_i(r\theta)|^2\,d\theta
\longrightarrow
\left(\frac{n}{H(p)}\right)^{2n-2}r^{1-n}
\quad\text{in }C^1([1/2,2]).
\]
By Lemma~\ref{second scale reduction}, the derivative of the left-hand side
vanishes at $r=1$. Passing to the limit gives
\[
0=
\left.\frac{d}{dr}
\left[
\left(\frac{n}{H(p)}\right)^{2n-2}r^{1-n}
\right]\right|_{r=1}
=-(n-1)\left(\frac{n}{H(p)}\right)^{2n-2}<0,
\]
a contradiction. Hence $p$ is an isolated simple blow-up point.
\end{proof}
\section{Proof of Theorem~\ref{thm:intro-compactness}}\label{Proof of Theorem 1.3}
Throughout Subsections~\ref{Sharp decay estimates} and~\ref{Green convergence}, we retain the hypotheses
and notation of Theorem~\ref{thm:intro-threshold} and assume that
$\|\psi_i\|_{L^\infty(\mathbb S^n)}\to\infty$.
\subsection{Sharp decay estimates}\label{Sharp decay estimates}
For the estimates in this subsection, we additionally assume
that $\widetilde H_i\in C^1(B_2)$ and
\[
\|\nabla\widetilde H_i\|_{L^\infty(B_2)}\le 2A_2.
\]
Together with the uniform positivity and boundedness of $H_i$,
this gives \eqref{eq:euclidean-coefficient-bounds} after choosing $A_1$ sufficiently large.

\begin{equation}\label{eq:euclidean-coefficient-bounds}
A_1^{-1}\le \widetilde H_i(y)\le A_1
\quad (y\in B_2),
\qquad
\|\nabla\widetilde H_i\|_{L^\infty(B_2)}
\le 2A_2.
\end{equation}
The following inequality will be used
to derive the decay estimate in Lemma~\ref{lem:local-green}.
\begin{lem}\label{exponential estimates}
Let $a>0$, let $v(t)\ge0$ be continuous on $[t_0,t_1]$, and let
$f(t)\in C^2([t_0,t_1])$ satisfy
\[
 f(t)>0,\qquad f'(t)\le0,\qquad f''(t)\ge(4a^2-4av(t))f(t).
\]
Then, for $t_0\le t\le t_1$,
\begin{equation}\label{eq:dirac-green-exact-scalar}
 f(t)\le
 \exp\left(1+2\int_{t_0}^{t_1}v(s)\,\mathrm{d} s\right)
 f(t_0)e^{-2a(t-t_0)}.
\end{equation}
\end{lem}
\begin{proof}
Put \(g(t)=(2a+\frac{f'(t)}{f(t)})_+\), then
 $g(t)$ is absolutely continuous and $0\le g(t)\le2a$. 
Where $g(t)>0$, 
\[
 g'(t)=\frac{f''(t)}{f(t)}-\left(\frac{f'(t)}{f(t)}\right)^2\ge4a^2-4av(t)-(g(t)-2a)^2\ge2ag(t)-4av(t).
\]Where $g(t)=0$, we have $g'(t)=0\ge -4av(t)$ almost everywhere.
Hence $g'(t)\ge 2ag(t)-4av(t)$ holds almost everywhere on $[t_0,t_1]$. Thus
\[
 (e^{-2at}g(t))'\ge-4ae^{-2at}v(t).
\]
Integrating from $t$ to $t_1$, we have
\begin{equation}\label{eq:dirac-green-deficit-bound}
 g(t)\le e^{-2a(t_1-t)}g(t_1)
       +4a\int_t^{t_1} e^{-2a(s-t)}v(s)\,\mathrm{d} s.
\end{equation}
The first term has integral at most $1$ because $g(t_1)\le2a$. For the second,
Tonelli's theorem gives
\begin{align*}
 &4a\int_{t_0}^t\int_{\tau}^{t_1} e^{-2a(s-\tau)}v(s)\,\mathrm{d} s\,\mathrm{d} \tau\\
 &\quad=4a\int_{t_0}^{t_1} v(s)
                \int_{t_0}^{\min\{t,s\}}e^{-2a(s-\tau)}\,\mathrm{d} \tau\,\mathrm{d} s
 \le2\int_{t_0}^{t_1}v(s)\,\mathrm{d} s.
\end{align*}
Consequently
\[
 \int_{t_0}^t g(\tau)\,\mathrm{d} \tau\le1+2\int_{t_0}^{t_1}v(s)\,\mathrm{d} s.
\]
Since $\frac{f'(t)}{f(t)}\le g(t)-2a$, we obtain
\[
 \log\frac{f(t)}{f(t_0)}
  =\int_{t_0}^t \frac{f'(\tau)}{f(\tau)}\,\mathrm{d} \tau
  \le-2a(t-t_0)+1+2\int_{t_0}^{t_1}v(s)\,\mathrm{d} s.
\]
Exponentiating completes the proof.\end{proof}
We next establish a pointwise decay estimate under
the assumption of isolated simple blow-up.
\begin{lem}
\label{lem:local-green} 
Suppose that $\widetilde H_i\in C^1(B_2)$ satisfy \eqref{eq:euclidean-coefficient-bounds},
and that $\varphi_i$ are weak solutions of 
\[
D_{g_{{\mathbb R^n}}}\varphi_i
=\widetilde H_i|\varphi_i|^{2^*-2}\varphi_i
\qquad\text{in }B_2.
\]
Assume that $0$ is an isolated simple blow-up point of
$\{\varphi_i\}$, with the points $y_i\to0$ chosen in
Subsection~\ref{Blow-up rescaling near a concentration point}, and that, for some constant $A_3>0$,
\begin{equation}\label{isolation bound 2}
|y-y_i|^{\frac{n-1}{2}}|\varphi_i(y)|\le A_3,
\qquad y\in B_2\setminus\{y_i\}.
\end{equation}
Assume also that the corresponding spherical spinors
$\psi_i$ satisfy the energy bound \eqref{eq:first-energy-bound}.
Then there exists $C=C(n,A_1,A_2,A_3,\rho)>0$ such that,
for all sufficiently large $i$,
\[
|\varphi_i(y)|+|y-y_i|\,|\nabla\varphi_i(y)|
\le C|\varphi_i(y_i)|^{-1}|y-y_i|^{1-n},
\qquad 0<|y-y_i|\le1,
\]
where $\rho$ is the constant in Definition \ref{def:intro-simple}.
\end{lem}

\begin{proof}
  By decreasing $\rho$ if necessary, we may assume that
$0<\rho<1$.
  Since $y_i\to0$, we may assume that $|y_i|<1/2$.
Then $B_{3/2}(y_i)\subset B_2$.
Using the cylindrical transformation from Subsection~\ref{Conformal covariance and stereographic projection}
centred at $y_i$, define
\[
U_i(t,\theta)
=e^{\frac{n-1}{2}t}
\beta\bigl(\varphi_i(y_i+e^t\theta)\bigr),
\qquad
(t,\theta)\in
\left(-\infty,\log\frac32\right)\times\mathbb S^{n-1}.
\]
By conformal covariance and \eqref{isolation bound 2},
 \begin{equation}\label{eq:U_i}
D_{\mathrm{cyl}}U_i
=\widetilde H_i(y_i+e^t\theta)|U_i|^{2^*-2}U_i,
\qquad |U_i|\le A_3.
 \end{equation}
For $h>0$, write
\[
Q_h(t_*)=(t_*-h,t_*+h)\times\mathbb S^{n-1}.
\]
We first prove the following claim.

\smallskip
\noindent\textit{Claim.}
Fix $h>0$. For every slab
$Q_h(t_*)\subset(-\infty,\log(\frac{3}{2}))\times\mathbb S^{n-1}$,
we have
\begin{equation}\label{Local elliptic estimate}
\|U_i\|_{W^{2,2n}(Q_{h/2}(t_*))}
+\|U_i\|_{C^{1,1/2}(\overline{Q_{h/2}(t_*)})}
\le C_h\|U_i\|_{L^2(Q_h(t_*))},
\end{equation}
where $C_h$ depends only on $n,h,A_1,A_2,A_3$.

Indeed, \eqref{eq:U_i} gives
$|D_{\mathrm{cyl}}U_i|\le C|U_i|$.
Thus, by the interior elliptic estimate for
$D_{\mathrm{cyl}}$, for $0<s<S\leq h$ and $1<p<\infty$,
\begin{equation}\label{eq:dirac-green-linear-step}
\|U_i\|_{W^{1,p}(Q_s(t_*))}
\le C_{p,s,S}\|U_i\|_{L^p(Q_S(t_*))}.
\end{equation}
To specify every exponent and every domain, define
\begin{equation}\label{eq:dirac-green-ladder}
 p_j=\frac{2n}{n-j}\quad(0\le j\le n-1),\qquad
 h_j=h\left(1-\frac{j}{2(n+1)}\right)\quad(0\le j\le n+1).
\end{equation}
Thus $p_0=2$, $p_{n-2}=n$, $p_{n-1}=2n$, and $h_{n+1}=h/2$.
We claim inductively that
\begin{equation}\label{eq:dirac-green-ladder-induction}
 \left\lVert U_i\right\rVert_{L^{p_j}( Q_{h_j}(t_*))}\le C_j \left\lVert U_i\right\rVert_{L^{2}( Q_{h}(t_*))},
 \qquad 0\le j\le n-1.
\end{equation}
The case $j=0$ is obvious.  Suppose that \eqref{eq:dirac-green-ladder-induction}
holds for some $j\in\{0,\ldots,n-2\}$.
We prove it for $j+1$. \eqref{eq:dirac-green-linear-step} gives
\begin{equation}\label{eq:dirac-green-ladder-W1}
 \left\lVert U_i\right\rVert_{W^{1,p_j}( Q_{h_{j+1}}(t_*))}\le C \left\lVert U_i\right\rVert_{L^{p_j}( Q_{h_j}(t_*))}\leq C \left\lVert U_i\right\rVert_{L^{2}( Q_{h}(t_*))}.
\end{equation}
When $j<n-2$, $p_{j+1}<p_j^*:=\frac{np_j}{n-p_j}$. Hence H\"older
inequality and Sobolev embedding imply that \[\left\lVert U_i\right\rVert_{L^{p_j+1}( Q_{h_{j+1}}(t_*))}\leq C\left\lVert U_i\right\rVert_{L^{p_j^*}( Q_{h_{j+1}}(t_*))}\leq C\left\lVert U_i\right\rVert_{W^{1,p_j}( Q_{h_{j+1}}(t_*))}\leq C \left\lVert U_i\right\rVert_{L^{2}( Q_{h}(t_*))}.\] 
At the last step $j=n-2$, we use the Sobolev embedding
\(
 W^{1,n}( Q_{h_{n-1}}(t_*))
 \hookrightarrow L^{2n}( Q_{h_{n-1}}(t_*)).
\)
This proves that
\[
 \left\lVert U_i\right\rVert_{L^{2n}( Q_{h_{n-1}}(t_*))}\le C\left\lVert U_i\right\rVert_{L^{2}( Q_{h}(t_*))}.
\]Applying the interior elliptic estimate once more,
with $p=2n$, yields
\[
\|U_i\|_{W^{1,2n}(Q_{h_n}(t_*))}
\le C\|U_i\|_{L^{2n}(Q_{h_{n-1}}(t_*))}
\le C\|U_i\|_{L^2(Q_h(t_*))}.
\]
Since $2^*-2>0$, the map $z\mapsto |z|^{2^*-2}z$ is continuously
differentiable in the real sense.
Thus, \eqref{eq:U_i} gives
\[
\bigl|\nabla(|U_i|^{2^*-2}U_i)\bigr|
\le (2^*-1)A_3^{2^*-2}|\nabla U_i|
\qquad\text{a.e.}
\]
Using the equation and the uniform $C^1$ bounds for
$\widetilde H_i(y_i+e^t\theta)$ on $Q_h(t_*)$, we obtain
\[
\|D_{\mathrm{cyl}}U_i\|_{W^{1,2n}(Q_{h_n}(t_*))}
\le C\|U_i\|_{W^{1,2n}(Q_{h_n}(t_*))}
\le C\|U_i\|_{L^2(Q_h(t_*))}.
\]
Since $h_{n+1}=h/2<h_n$, a further interior elliptic
estimate, followed by the Morrey embedding
$W^{2,2n}\hookrightarrow C^{1,1/2}$, yields
\[
\|U_i\|_{W^{2,2n}(Q_{h/2}(t_*))}
+
\|U_i\|_{C^{1,1/2}(\overline{Q_{h/2}(t_*)})}
\le C_h\|U_i\|_{L^2(Q_h(t_*))}.
\]
In particular, for every such slab $Q_h(t)$, we have
\begin{equation}\label{local-estimate}
\|U_i\|_{C^1(\overline{Q_{h/2}(t)})}
\le C_h\|U_i\|_{L^2(Q_h(t))}=C_h
\left(
\int_{t-h}^{t+h}
\|U_i(s,\cdot)\|_{L^2(\mathbb S^{n-1})}^2\,ds
\right)^{1/2}.
\end{equation}
We now use this estimate to study the decay of $U_i$.
By \eqref{local-estimate}, we have $U_i\in C^1_{\mathrm{loc}}$. Define
\begin{equation}\label{eq:dirac-green-f-w}
f_i(t)=\int_{\mathbb S^{n-1}}|U_i(t,\theta)|^2\,d\theta.
\end{equation}
Since $\beta$ is fibrewise isometric,
\begin{equation}\label{eq:cyl-f-w}
f_i(t)
=e^{(n-1)t}\int_{\mathbb S^{n-1}}
|\varphi_i(y_i+e^t\theta)|^2\,d\theta
=|\mathbb S^{n-1}|\,\bar w_i(e^t).
\end{equation}
Set
\[
V_i(t,\theta)
=\widetilde H_i(y_i+e^t\theta)|U_i(t,\theta)|^{2^*-2}.
\]
Equation~\eqref{eq:U_i} and the product decomposition give
\[
\nabla_{\partial_t}U_i
=\partial_t\cdot D_{\mathrm{vert}}U_i
-\partial_t\cdot(V_iU_i).
\]
Since $V_i$ is real-valued and Clifford multiplication
by $\partial_t$ is skew-adjoint,
\[
f_i'(t)
=2(\nabla_{\partial_t}U_i,U_i)_{L^2}
=2(\partial_t\cdot D_{\mathrm{vert}}U_i,U_i)_{L^2}.
\]
The product structure gives
\[
(\partial_t\cdot D_{\mathrm{vert}})^*
=-D_{\mathrm{vert}}(\partial_t\cdot)
=\partial_t\cdot D_{\mathrm{vert}}.
\]
Using this self-adjointness, the independence of the operator
from $t$, and the preceding $C^1$ regularity, we obtain
\[
\begin{aligned}
\frac{f_i'(t+\tau)-f_i'(t)}{\tau}
&=
2\left(
\partial_t\cdot D_{\mathrm{vert}}
\bigl(U_i(t+\tau)+U_i(t)\bigr),
\frac{U_i(t+\tau)-U_i(t)}{\tau}
\right)_{L^2}\\
&\longrightarrow
4(\partial_t\cdot D_{\mathrm{vert}}U_i(t),
\nabla_{\partial_t}U_i(t))_{L^2}
\qquad\text{as }\tau\to0.
\end{aligned}
\]
Substituting the expression for $\nabla_{\partial_t}U_i$
therefore yields
\begin{equation}\label{eq:cyl-f-second}
f_i''(t)
=4\|D_{\mathrm{vert}}U_i\|_{L^2}^2
-4(D_{\mathrm{vert}}U_i,V_iU_i)_{L^2}.
\end{equation}
The right-hand side is continuous, so
$f_i\in C^2_{\mathrm{loc}}$.

We next use the limiting bubble to obtain initial bounds for $f_i$.

By \eqref{eq:cyl-f-w} and the simple blow-up assumption, $f_i$ has exactly
one critical point in $(-\infty,\log\rho)$.
By Lemma~\ref{lem:rescaled-limit}, after passing to a subsequence,
$\Phi_i\to\Phi$ in $C^{0,\gamma}_{\mathrm{loc}}(\mathbb R^n)$
for every $0<\gamma<1$. Applying the local elliptic argument used to prove \eqref{Local elliptic estimate} to the rescaled
equation, we obtain uniform $W^{2,2n}_{\mathrm{loc}}$ bounds for $\Phi_i$.
By compact Sobolev embedding, 
$\Phi_i\to\Phi$ in $C^1_{\mathrm{loc}}(\mathbb R^n)$. By Corollary~\ref{cor:first-bubble},
\[
|\Phi(y)|=(1+c^2|y|^2)^{-\frac{n-1}{2}},
\qquad c=\frac{H(p)}{n}.
\]
The definitions of $U_i$ and $\Phi_i$, together with the
fibrewise isometry of $\beta$, yield
\[
f_i(\log(\lambda_i r))
=r^{n-1}\int_{\mathbb S^{n-1}}|\Phi_i(r\theta)|^2\,d\theta
\longrightarrow
|\mathbb S^{n-1}|
\left(\frac{r}{1+c^2r^2}\right)^{n-1}
\]
in $C^1_{\mathrm{loc}}((0,\infty))$.

The limit is strictly decreasing for $r>1/c$.
Choose a large fixed $R$ with $cR>1$ and set
$t_i=\log(R\lambda_i)$.
For all sufficiently large $i$, we then have $t_i<\log\rho$ and
\[
0<f_i(t_i)\le CR^{-(n-1)},
\qquad f_i'(t_i)<0,
\]where $C$ is independent of $R$ and $i$.

For each fixed $i$, $\lim\limits_{t\to-\infty}f_i(t)=0$.
Together with the preceding inequalities, this implies that
$f_i$ has a critical point in $(-\infty,t_i)$.
By uniqueness of the critical point,
\[
f_i'(t)<0,\qquad f_i(t)>0,
\qquad t_i\le t<\log\rho.
\]
For $[t-h,t+h]\subset[t_i,\log\rho)$, monotonicity gives
\[
\|U_i\|_{L^2(Q_h(t))}^2
=\int_{t-h}^{t+h}f_i(s)\,ds
\le 2h\,f_i(t-h).
\]
Combining this with \eqref{Local elliptic estimate}, we obtain.
\begin{equation}\label{f_i bounded}
\|U_i\|_{C^1(\overline{Q_{h/2}(t)})}
\le C_h f_i(t-h)^{1/2}.
\end{equation}
We next estimate $V_i(t,\theta)=\widetilde H_i(y_i+e^t\theta)
|U_i(t,\theta)|^{2^*-2}$. Fix $h>0$ and set
\[
v_i(t)=\|V_i(t,\cdot)\|_{L^\infty(\mathbb S^{n-1})}.
\]
By \eqref{eq:euclidean-coefficient-bounds}, \eqref{f_i bounded} and monotonicity,
\[
\begin{aligned}
v_i(t)
&\le A_1
\|U_i(t,\cdot)\|_{L^\infty(\mathbb S^{n-1})}^{2/(n-1)}
\le C_h f_i(t-h)^{1/(n-1)}\\
&\le C_h f_i(t_i)^{1/(n-1)}
\le \frac{C_h}{R},
\qquad t_i+h\le t\le\log\rho-h.
\end{aligned}
\]
Enlarge the fixed $R$ so that $\frac{C_h}{R}\le \frac{n-1}{4}$,
and then take $i$ sufficiently large that
$t_i+3h<\log\rho$.
\\
The identity for $f_i''$, the Cauchy--Schwarz inequality
and the spectral gap give, on $[t_i+h,\log\rho-h]$,
\[
\begin{aligned}
f_i''(t)
&\ge 4\|D_{\mathrm{vert}}U_i\|_2^2
-4v_i(t)\|D_{\mathrm{vert}}U_i\|_2 f_i(t)^{1/2}\\
&\ge 4\left(1-\frac{2v_i(t)}{n-1}\right)
\|D_{\mathrm{vert}}U_i\|_2^2\\
&\ge \bigl((n-1)^2-2(n-1)v_i(t)\bigr)f_i(t)
\ge \frac{(n-1)^2}{2}f_i(t).
\end{aligned}
\]
Applying Lemma~\ref{exponential estimates} with
$a=\frac{n-1}{2\sqrt{2}}$ and $v=0$, and using
$f_i(t_i+h)\le f_i(t_i)$, we obtain
\[
f_i(t)\le
CR^{-(n-1)}
\exp\left(-\frac{n-1}{\sqrt{2}}(t-t_i-h)\right),
\qquad t_i+h\le t\le\log\rho-h.
\]
This decay gives a uniform bound for the integral of $v_i$.
Indeed, for $t_i+2h\le t\le\log\rho-h$,
\[
v_i(t)
\le C_h f_i(t-h)^{1/(n-1)}
\le \frac{C_h}{R}
\exp\left(-\frac{t-t_i-2h}{\sqrt{2}}\right).
\]
Together with $v_i\le \frac{C_h}{R}$ on $[t_i+h,t_i+2h]$,
this yields
\begin{equation}\label{nonlinear estimation}
\begin{aligned}
\int_{t_i+h}^{\log\rho-h}v_i(t)\,dt
&\le \frac{C_hh}{R}
+\frac{C_h}{R}
\int_{t_i+2h}^{\log\rho-h}
\exp\left(-\frac{t-t_i-2h}{\sqrt{2}}\right)\,dt\\
&\le \frac{C_h}{R}
\left(h+\int_0^\infty e^{-s/\sqrt{2}}\,ds\right)
\le \frac{C_h}{R}.
\end{aligned}
\end{equation}
Using
\(
f_i''\ge \bigl((n-1)^2-2(n-1)v_i\bigr)f_i,
\) we apply Lemma~\ref{exponential estimates} again  with $a=\frac{n-1}{2}$ and $v=v_i$.
\eqref{nonlinear estimation} gives
\[
\begin{aligned}
f_i(t)
&\le
\exp\left(1+2\int_{t_i+h}^{\log\rho-h}v_i(s)\,ds\right)
f_i(t_i+h)e^{-(n-1)(t-t_i-h)}\\
&\le C_hR^{-(n-1)}e^{-(n-1)(t-t_i-h)}.
\end{aligned}
\]
Since $t_i=\log(R\lambda_i)$ and
$\lambda_i^{n-1}=|\varphi_i(y_i)|^{-2}$, we obtain
\[
f_i(t)\le
C_h|\varphi_i(y_i)|^{-2}e^{-(n-1)t},
\qquad t_i+h\le t\le\log\rho-h.
\]
Combining this with \eqref{f_i bounded} yields
\[
\|U_i\|_{C^1(\overline{Q_{h/2}(t)})}
\le C_h f_i(t-h)^{1/2}
\le C_h|\varphi_i(y_i)|^{-1}e^{-\frac{n-1}{2}t},
\qquad t_i+2h\le t\le\log\rho-h.
\]
Write $y=y_i+r\theta$, where $r=e^t=|y-y_i|$.
By the definition of $U_i$,
\[
\varphi_i(y)
=r^{-\frac{n-1}{2}}
\beta^{-1}\bigl(U_i(\log r,\theta)\bigr).
\]
The conformal transformation formula for the spin connection
therefore gives
\[
\begin{aligned}
|\varphi_i(y)|+r|\nabla\varphi_i(y)|
&\le C_n r^{-\frac{n-1}{2}}
\left(
|U_i(\log r,\theta)|
+|\nabla^{\mathrm{cyl}}U_i(\log r,\theta)|
\right)\\
&\le C_h|\varphi_i(y_i)|^{-1}r^{1-n},
\end{aligned}
\]
whenever
\[
e^{2h}R\lambda_i\le r\le\rho e^{-h}.
\]

It remains to consider $0<r\le e^{2h}R\lambda_i$.
The uniform $C^1$ bounds for $\Phi_i$ on the fixed ball
$B_{e^{2h}R}$ give
\[
|\varphi_i(y)|+r|\nabla\varphi_i(y)|
\le C_{h,R}\lambda_i^{-\frac{n-1}{2}}.
\]
Since $r\le e^{2h}R\lambda_i$,
\[
\lambda_i^{-\frac{n-1}{2}}
\le (e^{2h}R)^{n-1}
\lambda_i^{\frac{n-1}{2}}r^{1-n}
=(e^{2h}R)^{n-1}
|\varphi_i(y_i)|^{-1}r^{1-n}.
\]
Combining the two regions, with $h$ and $R$ now fixed,
we conclude that
\begin{equation}\label{varphi_i inner}
|\varphi_i(y)|+|y-y_i||\nabla\varphi_i(y)|
\le C|\varphi_i(y_i)|^{-1}|y-y_i|^{1-n},
\qquad 0<|y-y_i|\le\rho e^{-h}.
\end{equation}

We finally extend the estimate to $0<|y-y_i|\le1$.
Set $r_0=\rho e^{-h}$. 
Choose $\chi\in C^\infty(\mathbb S^n)$, with $0\le\chi\le1$,
such that
\[
\chi=0\quad\text{on }\pi_p^{-1}(B_{r_0/4}(0)),
\qquad
\chi=1\quad\text{on }
\mathbb S^n\setminus\pi_p^{-1}(B_{r_0/2}(0)).
\]
For all sufficiently large $i$, $|y_i|<r_0/8$, and hence
\[
\frac{r_0}{8}
\le|\pi_p(x)-y_i|
\le\frac{5r_0}{8}<r_0,
\qquad x\in\operatorname{supp}(\nabla\chi).
\]Since $\chi$ is fixed, \eqref{varphi_i inner}
and the conformal identification give
\[
\|\nabla\chi\cdot\psi_i\|_
{L^{\frac{2n}{n+1}}(\mathbb S^n)}
\le C|\varphi_i(y_i)|^{-1}.
\]
Since $p\notin\operatorname{supp}\chi$, Theorem~\ref{thm:intro-threshold} and the
uniform bound on $H_i$ imply
\[
\|H_i|\psi_i|^{2^*-2}\|_
{L^n(\operatorname{supp}\chi)}
\le C\left(
\int_{\operatorname{supp}\chi}
|\psi_i|^{2^*}\,dv_{g_0}
\right)^{1/n}
=o(1).
\]
Using
the invertibility of $D_{g_0}$, the global elliptic estimate
and Sobolev embedding yield
\[
\begin{aligned}
\|\chi\psi_i\|_{L^{2^*}(\mathbb S^n)}
&\le C\|D_{g_0}(\chi\psi_i)\|_
{L^{\frac{2n}{n+1}}(\mathbb S^n)}\\
&\le
C\bigl\|H_i|\psi_i|^{2^*-2}\bigr\|_
{L^n(\operatorname{supp}\chi)}
\|\chi\psi_i\|_{L^{2^*}(\mathbb S^n)}
+C|\varphi_i(y_i)|^{-1}.
\end{aligned}
\]
Absorbing the first term on the right, we obtain
\[
\|\chi\psi_i\|_{L^{2^*}(\mathbb S^n)}
\le C|\varphi_i(y_i)|^{-1}.
\]

On the enlarged annulus
\[
\left\{y:\frac{3r_0}{4}<|y-y_i|<\frac54\right\}
\subset B_2,
\]
we have, for all sufficiently large $i$,
\[
\chi\circ\pi_p^{-1}=1,
\qquad
|\varphi_i|
\le A_3\left(\frac{3r_0}{4}\right)^{-\frac{n-1}{2}}.
\]
Repeating the local elliptic argument above, and using
H\"older's inequality and conformal invariance of the
$L^{2^*}$ norm, we obtain
\[
\begin{aligned}
\sup_{r_0\le|y-y_i|\le1}
\bigl(|\varphi_i(y)|+|\nabla\varphi_i(y)|\bigr)
&\le C\|\varphi_i\|_
{L^2(\{3r_0/4<|y-y_i|<5/4\})}\\
&\le C\|\chi\psi_i\|_{L^{2^*}(\mathbb S^n)}\\
&\le C|\varphi_i(y_i)|^{-1}.
\end{aligned}
\]
Since $|y-y_i|\le1$, this implies
\[
|\varphi_i(y)|+|y-y_i||\nabla\varphi_i(y)|
\le C|\varphi_i(y_i)|^{-1}|y-y_i|^{1-n},
\qquad r_0\le|y-y_i|\le1.
\]
Combining this with the estimate for $0<|y-y_i|\le r_0$
completes the proof.
\end{proof}
The preceding decay estimate allows us to identify the limit
of $|\varphi_i(y_i)|\varphi_i(y)$ away from the origin.
\subsection{Green convergence}\label{Green convergence}
\begin{lem}
\label{lem:green-conv-local}
Under the assumptions of Lemma~\ref{lem:local-green}, after passing to a
subsequence, we have
\[
|\varphi_i(y_i)|\varphi_i(y)\to -\frac1{\omega_{n-1}}\frac{y}{|y|^n}\cdot\eta
\quad
\text{in }C^1_{\mathrm{loc}}(B_1\setminus\{0\}),
\]
where \[
\eta=\omega_{n-1}
\left(\frac{n}{H(p)}\right)^{n-1}\Phi(0)\ne0,
\]
$\Phi$ is the rescaled limit in Lemma~\ref{lem:rescaled-limit}.
\end{lem}
\begin{proof}Set $\zeta_i(y):=|\varphi_i(y_i)|\varphi_i(y)$. Then
\begin{equation}\label{eq:green-normalized}
D_{g_{\mathbb R^n}}\zeta_i
=
\widetilde H_i(y)|\varphi_i(y_i)|^{-(2^*-2)}
|\zeta_i|^{2^*-2}\zeta_i
\quad\text{in }B_2.
\end{equation}
By Lemma~\ref{lem:local-green}, \eqref{eq:green-normalized} and interior elliptic estimates,
$\zeta_i$ are uniformly bounded in
$W^{2,2n}_{\mathrm{loc}}(B_1\setminus\{0\})$.
Thus, after passing to a subsequence,
\begin{equation}\label{local convergence away 0}
\zeta_i\longrightarrow\zeta
\quad\text{in }C^1_{\mathrm{loc}}(B_1\setminus\{0\}).
\end{equation}
Since $|\varphi_i(y_i)|\to\infty$, passing to the limit
in \eqref{eq:green-normalized} yields
\[
D_{g_{\mathbb R^n}}\zeta=0
\quad\text{in }B_1\setminus\{0\}.
\]We now identify the singularity of $\zeta$ at the origin. By Lemma~\ref{lem:local-green}, for every sufficiently small $\varepsilon>0$
and all sufficiently large $i$,
\[
\int_{B_\varepsilon}|\zeta_i|\,dy
\le C\int_{B_{2\varepsilon}(y_i)}
|y-y_i|^{1-n}\,dy
\le C\varepsilon.
\]
Together with \eqref{local convergence away 0},
this implies
\begin{equation}\label{L^1 convergence}
\zeta_i\longrightarrow\zeta
\quad\text{in }L^1_{\mathrm{loc}}(B_1).
\end{equation}
Let $\phi\in C_c^\infty(B_1,\Sigma\mathbb R^n)$.
By the formal self-adjointness of $D_{g_{\mathbb{R}^n}}$,
\begin{equation}\label{eq:green-weak}
\int_{B_1}\langle\zeta_i,D_{g_{\mathbb R^n}}\phi\rangle\,dy
=
|\varphi_i(y_i)|
\int_{B_1}\widetilde H_i
\langle|\varphi_i|^{2^*-2}\varphi_i,\phi\rangle\,dy.
\end{equation}
Fix $R>0$. For sufficiently large $i$, we split the
right-hand side into integrals over $B_{R\lambda_i}(y_i)$
and $B_1\setminus B_{R\lambda_i}(y_i)$. 

For the first integral, changing variables
$y=y_i+\lambda_i z$ and using Lemma~\ref{lem:rescaled-limit}, we obtain
\begin{equation}\label{first integral}
\begin{aligned}
|\varphi_i(y_i)|
\int_{B_{R\lambda_i}(y_i)}
\widetilde H_i(y)
\langle|\varphi_i|^{2^*-2}\varphi_i,\phi(y)\rangle\,dy&=
\int_{B_R}\widetilde H_i(y_i+\lambda_i z)
\left\langle
|\Phi_i(z)|^{2^*-2}\Phi_i(z),
\phi(y_i+\lambda_i z)
\right\rangle\,dz\\
\quad\longrightarrow &
H(p)\int_{B_R}
\left\langle
|\Phi(z)|^{2^*-2}\Phi(z),\phi(0)
\right\rangle\,dz.
\end{aligned}
\end{equation}
By Corollary~\ref{cor:first-bubble}, letting $R\to\infty$ gives
$\langle\eta,\phi(0)\rangle$, where, by the oddness
of the Clifford term,
\[
\begin{aligned}
\eta
:=H(p)\int_{\mathbb R^n}|\Phi|^{2^*-2}\Phi\,dz &=H(p)\left[
\int_{\mathbb R^n}
\left(1+\frac{H(p)^2}{n^2}|z|^2\right)^{-\frac{n+2}{2}}
\,dz
\right]\Phi(0)\\
&=\omega_{n-1}
\left(\frac{n}{H(p)}\right)^{n-1}\Phi(0)\ne0,
\end{aligned}
\]
since $H(p)>0$ and $|\Phi(0)|=1$.

For the second integral, we have
\begin{equation}\label{second integral}
\begin{aligned}
|\varphi_i(y_i)|
\left|
\int_{B_1\setminus B_{R\lambda_i}(y_i)}
\widetilde H_i
\langle|\varphi_i|^{2^*-2}\varphi_i,\phi\rangle\,dy
\right|&\le
C\|\phi\|_{L^\infty}|\varphi_i(y_i)|
\int_{\operatorname{supp}\phi\setminus B_{R\lambda_i}(y_i)}
|\varphi_i|^{2^*-1}\,dy\\
&\le
C\|\phi\|_{L^\infty}\lambda_i
\int_{R\lambda_i}^{1}r^{-2}\,dr
\le \frac{C\|\phi\|_{L^\infty}}{R},
\end{aligned}
\end{equation}
where the penultimate inequality follows from Lemma~\ref{lem:local-green}
and $\lambda_i=|\varphi_i(y_i)|^{-(2^*-2)}$ and
$\operatorname{supp}\phi\subset B_1(y_i)$ for all sufficiently
large $i$.

Combining \eqref{L^1 convergence}, \eqref{first integral}, \eqref{second integral},
we let first $i\to\infty$ and then $R\to\infty$
in \eqref{eq:green-weak} to obtain
\[
\int_{B_1}
\langle\zeta,D_{g_{\mathbb{R}^n}}\phi\rangle\,dy
=
\langle\eta,\phi(0)\rangle.
\]
Thus,
\[
D_{g_{\mathbb R^n}}\zeta=\eta\delta_0
\quad\text{in }\mathcal D'(B_1).
\]By \cite{BE2006}, the Euclidean Dirac
fundamental solution
\[
\Gamma(y)\eta
=-\frac{1}{\omega_{n-1}}\frac{y}{|y|^n}\cdot\eta
\]
satisfies
\[
D_{g_{\mathbb R^n}}(\Gamma\eta)=\eta\delta_0
\quad\text{in }\mathcal D'(\mathbb R^n).
\]
Thus $h:=\zeta-\Gamma\eta$ satisfies
$D_{g_{\mathbb{R}^n}}h=0$ distributionally in $B_1$.
Elliptic regularity gives
$h\in C^\infty(B_1,\Sigma\mathbb R^n)$.
On the other hand, taking $\mu_i=1$ in \eqref{eq:kelvin-second-scale}, we obtain
\[
|\varphi_i(y_i)|\varphi_i(y_i+y)
=\lambda_i^{1-n}\Phi_i\left(\frac{y}{\lambda_i}\right)\longrightarrow
-\left(\frac{n}{H(p)}\right)^{n-1}
\frac{y}{|y|^n}\cdot\Phi(0)
=\Gamma(y)\eta
\]
in $C^0_{\mathrm{loc}}(\mathbb R^n\setminus\{0\})$. Since $\zeta_i\to\zeta$ locally uniformly in
$B_1\setminus\{0\}$ and $y_i\to0$, we also have
\[
|\varphi_i(y_i)|\varphi_i(y_i+y)
=\zeta_i(y_i+y)
\longrightarrow
\zeta(y)=\Gamma(y)\eta+h(y).
\]
Comparing the two limits gives $h=0$ in
$B_1\setminus\{0\}$. By continuity, $h\equiv0$ in $B_1$.
This completes the proof.
\end{proof}
\subsection{Pohozaev identities and compactness}

We combine the sharp decay estimate and the Green convergence
with the Pohozaev identity to obtain a second-order condition
at the blow-up point.

\begin{prop}\label{prop:second-order-obstruction}
Assume the hypotheses of
Theorem~\ref{thm:intro-threshold}, and let $p$ be the
blow-up point obtained there.
Suppose that
\(
H_i\longrightarrow H\) in \(
C^2(U)
\)
for some neighborhood $U$ of $p$.
Then
\(
\nabla\widetilde H_i(y_i)=o(\lambda_i)\),
 \(
\operatorname{Hess}_{g_0}H(p)=0.
\)
\end{prop}

\begin{proof}
  Choose $s>0$ sufficiently small so that
$\pi_p^{-1}(\overline{B_{2s}})\subset U$
and the isolation estimate holds on $B_{2s}$.
By Theorem~\ref{thm:intro-simple}, $p$ is an isolated simple blow-up point.
Applying Lemma~\ref{lem:local-green} to
$s^{(n-1)/2}\varphi_i(s\,\cdot)$ and scaling back, we obtain
the sharp decay estimate for $0<|y-y_i|\le s$
.

Fix $0<r<s/4$. For all sufficiently large $i$,
Lemma~\ref{lem:local-green} gives
\begin{equation}\label{sharp decay estimate}
|\varphi_i(y)|+|y-y_i|\,|\nabla\varphi_i(y)|
\le
C\lambda_i^{\frac{n-1}{2}}|y-y_i|^{1-n},
\qquad 0<|y-y_i|\le 2r.
\end{equation}
Together with the local uniform bounds for $\Phi_i$,
this implies
\[
|\Phi_i(z)|\le C(1+|z|)^{1-n},
\qquad |z|<\frac{2r}{\lambda_i}.
\]
Since $n\ge3$,
\(
\int_{\mathbb R^n}
\frac{1+|z|^2}{(1+|z|)^{2n}}\,dz<\infty.
\)
The local convergence of $\Phi_i$ and the radiality
of $|\Phi|$ therefore give
\begin{equation}\label{eq:pohozaev-moments}
\begin{aligned}
\int_{B_{r/\lambda_i}}|\Phi_i|^{2^*}\,dz
&\longrightarrow
\int_{\mathbb R^n}|\Phi|^{2^*}\,dz>0,\\
\int_{B_{r/\lambda_i}}z|\Phi_i|^{2^*}\,dz
&\longrightarrow0,\\
\int_{B_{r/\lambda_i}}z_jz_k|\Phi_i|^{2^*}\,dz
&\longrightarrow c_\Phi\delta_{jk},
\end{aligned}
\end{equation}
where
\(
c_\Phi
:=
\frac1n\int_{\mathbb R^n}|z|^2|\Phi|^{2^*}\,dz>0.
\) The Green convergence in
Lemma~\ref{lem:green-conv-local} implies that
\begin{equation}\label{eq:pure-green-for-pohozaev}
\lambda_i^{-\frac{n-1}{2}}\varphi_i(y_i+z)
\longrightarrow -\left(\frac{n}{H(p)}\right)^{n-1}
\frac{z}{|z|^n}\cdot\Phi(0)
\quad\text{in }C^1_{\mathrm{loc}}(B_{2r}\setminus\{0\}).
\end{equation}
By the Pohozaev identity \cite{BE},
applied to $B_r(y_i)$,
we have
\begin{equation}\label{eq:dirac-local-pohozaev}
\begin{aligned}
\frac{n-1}{n}
\int_{B_r(y_i)}
X(\widetilde H_i)|\varphi_i|^{2^*}\,dy
={}&
\int_{\partial B_r(y_i)}
(\nu\cdot\nabla_X\varphi_i,\varphi_i)
\,d\sigma\\
&-\frac1n
\int_{\partial B_r(y_i)}
(X\cdot\nu)\widetilde H_i|\varphi_i|^{2^*}\,d\sigma,
\end{aligned}
\end{equation}
where $X=e_j$ $(j=1,\ldots,n)$ or $X(y)=y-y_i$,
and $\nu$ is the outward unit normal to
$\partial B_r(y_i)$.

Taking $X=e_j$ and using \eqref{sharp decay estimate} on
$\partial B_r(y_i)$ yields
\begin{equation}\label{eq:translation-pohozaev-bound}
\int_{B_r(y_i)}
\nabla\widetilde H_i(y)|\varphi_i(y)|^{2^*}\,dy
=
O(\lambda_i^{n-1}).
\end{equation}After the change of variables $y=y_i+\lambda_i z$,
Taylor expansion gives
\[
\nabla\widetilde H_i(y_i)
\int_{B_{r/\lambda_i}}|\Phi_i|^{2^*}\,dz+
\lambda_i D^2\widetilde H_i(y_i)
\int_{B_{r/\lambda_i}}z|\Phi_i|^{2^*}\,dz
+o(\lambda_i)
=
O(\lambda_i^{n-1}).
\]
By the $C^2$ convergence,
\eqref{eq:pohozaev-moments}, and $n\ge3$, we obtain
\begin{equation}\label{eq:gradient-small-scale}
\nabla\widetilde H_i(y_i)=o(\lambda_i).
\end{equation}

We now take $X=y-y_i$ in
\eqref{eq:dirac-local-pohozaev}.
On $\partial B_r(y_i)$, $\nabla_X=r\nabla_\nu$.
Moreover,
\[
\partial_\nu(\frac{z}{|z|^n}\cdot\Phi(0))=\frac{1-n}{r}(\frac{z}{|z|^n}\cdot\Phi(0))
\quad\text{on }\partial B_r.
\]
Thus, by \eqref{eq:pure-green-for-pohozaev},
\[
\begin{aligned}
&\lambda_i^{1-n}
\int_{\partial B_r(y_i)}
(\nu\cdot\nabla_X\varphi_i,\varphi_i)
\,d\sigma\longrightarrow
r\left(\frac{n}{H(p)}\right)^{2(n-1)}\int_{\partial B_r}
(\nu\cdot\partial_\nu(\frac{z}{|z|^n}\cdot\Phi(0)),\frac{z}{|z|^n}\cdot\Phi(0))\,d\sigma
=0.
\end{aligned}
\]
The nonlinear boundary term satisfies
\[
\int_{\partial B_r(y_i)}
\widetilde H_i|\varphi_i|^{2^*}\,d\sigma
=
O(\lambda_i^n).
\]
Consequently,
\begin{equation}\label{eq:dilation-pohozaev-small}
\begin{aligned}
&\int_{B_r(y_i)}
(y-y_i)\cdot\nabla\widetilde H_i(y)
|\varphi_i(y)|^{2^*}\,dy=
o(\lambda_i^{n-1})+O(\lambda_i^n)
=o(\lambda_i^2).
\end{aligned}
\end{equation}

On the other hand, Taylor expansion gives
\[
\begin{aligned}
\lambda_i^{-2}
\int_{B_r(y_i)}
(y-y_i)\cdot\nabla\widetilde H_i(y)
|\varphi_i(y)|^{2^*}\,dy&=
\frac{\nabla\widetilde H_i(y_i)}{\lambda_i}\cdot
\int_{B_{r/\lambda_i}}z|\Phi_i|^{2^*}\,dz\\
&+
\int_{B_{r/\lambda_i}}
z^T D^2\widetilde H_i(y_i)z
|\Phi_i|^{2^*}\,dz
+o(1)\longrightarrow
c_\Phi\Delta\widetilde H(0).
\end{aligned}
\]
Combining this with \eqref{eq:gradient-small-scale} and
\eqref{eq:dilation-pohozaev-small}, we obtain
\(
\Delta\widetilde H(0)=0.
\)
In the stereographic coordinates centered at $p$,
\(
\Delta\widetilde H(0)=4\Delta_{g_0}H(p),
\)
and hence $\Delta_{g_0}H(p)=0$.

Finally, Theorem~\ref{thm:intro-threshold} gives
$H(p)=H_{\max}$.
Thus $\operatorname{Hess}_{g_0}H(p)$ is negative
semidefinite and has zero trace, so
\[
\operatorname{Hess}_{g_0}H(p)=0.
\]
\end{proof}
\begin{proof}[Proof of Theorem~\ref{thm:intro-compactness}]
Suppose that no $\varepsilon,\delta,C>0$ give the
asserted bound. Taking
$\varepsilon=\delta=1/i$ and $C=i$, we obtain positive
functions $H_i\in C^2(\mathbb S^n)$ and weak solutions
$\psi_i$ such that
\[
\|H_i-H\|_{C^2(\mathbb S^n)}<\frac1i,
\qquad
\|\psi_i\|_{L^\infty(\mathbb S^n)}>i,
\]
and
\[
\int_{\mathbb S^n}|\psi_i|^{2^*}\,dv_{g_0}
\le
\left(\max_{\mathbb S^n}H_i\right)^{-n}
\left(\frac n2\right)^n\omega_n
+\frac1i.
\]
Since $H_i\to H$ uniformly,
\[
\limsup_{i\to\infty}
\int_{\mathbb S^n}|\psi_i|^{2^*}\,dv_{g_0}
\le
H_{\max}^{-n}
\left(\frac n2\right)^n\omega_n.
\]
Theorem~\ref{thm:intro-threshold} therefore gives,
after passing to a subsequence, a unique blow-up
point $p$ satisfying
\(
H(p)=H_{\max}.
\)
Since $H_i\to H$ in $C^2(\mathbb S^n)$,
Proposition~\ref{prop:second-order-obstruction} yields
\(
\Delta_{g_0}H(p)=0,
\)
contrary to the assumption.
This proves the uniform $L^\infty$ bound.
\end{proof}


\begin{thebibliography}{99}\small

\bibitem{A2003}
B.~Ammann,
\emph{A variational problem in conformal spin geometry},
Habilitationsschrift, Universit{\"a}t Hamburg, 2003.

\bibitem{Ammann2009}
B.~Ammann,
The smallest Dirac eigenvalue in a spin-conformal class and
cmc-immersions,
\emph{Commun. Anal. Geom.} \textbf{17} (2009), 429--479.
\href{https://doi.org/10.4310/CAG.2009.v17.n3.a2}
{doi:10.4310/CAG.2009.v17.n3.a2}.

\bibitem{BE2006}
B.~Ammann, E.~Humbert and B.~Morel,
Mass endomorphism and spinorial Yamabe type problems on
conformally flat manifolds,
\emph{Commun. Anal. Geom.} \textbf{14} (2006), 163--182.

\bibitem{BE}
B.~Ammann, E.~Humbert and M.~Ould Ahmedou,
An obstruction for the mean curvature of a conformal immersion
$S^n\to\mathbb R^{n+1}$,
\emph{Proc. Amer. Math. Soc.} \textbf{135} (2007),
no.~2, 489--493.

\bibitem{AmmannJammes}
B.~Ammann and P.~Jammes,
The supremum of first eigenvalues of conformally covariant
operators in a conformal class,
in \emph{Variational Problems in Differential Geometry},
London Mathematical Society Lecture Note Series, vol.~394,
Cambridge University Press, 2011, pp.~1--23.
\href{https://doi.org/10.1017/CBO9780511863219.002}
{doi:10.1017/CBO9780511863219.002}.

\bibitem{Baer1996}
C.~B{\"a}r,
The Dirac operator on space forms of positive curvature,
\emph{J. Math. Soc. Japan} \textbf{48} (1996),
no.~1, 69--83.

\bibitem{BF2020}
W.~Borrelli and R.~L.~Frank,
Sharp decay estimates for critical Dirac equations,
\emph{Trans. Amer. Math. Soc.} \textbf{373} (2020), 2045--2070.
\href{https://doi.org/10.1090/tran/7958}
{doi:10.1090/tran/7958}.

\bibitem{BMW2021}
W.~Borrelli, A.~Malchiodi and R.~Wu,
Ground state Dirac bubbles and Killing spinors,
\emph{Commun. Math. Phys.} \textbf{383} (2021), 1151--1180.
\href{https://doi.org/10.1007/s00220-021-04013-1}
{doi:10.1007/s00220-021-04013-1}.

\bibitem{BHMMM2015}
J.-P.~Bourguignon, O.~Hijazi, J.-L.~Milhorat et al.,
\emph{A spinorial approach to Riemannian and conformal geometry},
EMS Monographs in Mathematics,
European Mathematical Society, 2015.

\bibitem{CGS1989}
L.~A.~Caffarelli, B.~Gidas and J.~Spruck,
Asymptotic symmetry and local behavior of semilinear
elliptic equations with critical Sobolev growth,
\emph{Comm. Pure Appl. Math.} \textbf{42} (1989),
no.~3, 271--297.
doi:10.1002/cpa.3160420304

\bibitem{CCZ2026}
P.~Chen, Y.~Chen and M.~Zhu,
Energy quantization for Dirac systems over non-collapsed
degenerating Einstein manifolds,
arXiv:2606.27686, 2026.
\url{https://arxiv.org/abs/2606.27686}.

\bibitem{Friedrich1998}
T.~Friedrich,
On the spinor representation of surfaces in Euclidean $3$-space,
\emph{J. Geom. Phys.} \textbf{28} (1998),
no.~1--2, 143--157.
doi:10.1016/S0393-0440(98)00018-7.

\bibitem{F2000}
T.~Friedrich,
\emph{Dirac operators in Riemannian geometry},
Graduate Studies in Mathematics, vol.~25,
American Mathematical Society, 2000.


\bibitem{Isobe2011}
T.~Isobe,
Nonlinear Dirac equations with critical nonlinearities
on compact Spin manifolds,
\emph{J. Funct. Anal.} \textbf{260} (2011),
no.~1, 253--307.
\href{https://doi.org/10.1016/j.jfa.2010.09.008}
{doi:10.1016/j.jfa.2010.09.008}.

\bibitem{Isobe2013}
T.~Isobe,
A perturbation method for spinorial Yamabe type equations
on $S^m$ and its application,
\emph{Math. Ann.} \textbf{355} (2013),
no.~4, 1255--1299.

\bibitem{ISX2024}
T.~Isobe, Y.~Sire and T.~Xu,
Non-compactness results for the spinorial Yamabe-type problems
with non-smooth geometric data,
\emph{J. Funct. Anal.} \textbf{287} (2024), 110472.
\href{https://doi.org/10.1016/j.jfa.2024.110472}
{doi:10.1016/j.jfa.2024.110472}.


\bibitem{KMS2009}
M.~A.~Khuri, F.~C.~Marques and R.~M.~Schoen,
A compactness theorem for the Yamabe problem,
\emph{J. Differential Geom.} \textbf{81} (2009),
no.~1, 143--196.
doi:10.4310/jdg/1228400630.

\bibitem{Kenmotsu1979}
K.~Kenmotsu,
Weierstrass formula for surfaces of prescribed mean curvature,
\emph{Math. Ann.} \textbf{245} (1979),
no.~2, 89--99.
doi:10.1007/BF01428799.

\bibitem{KusnerSchmitt1996}
R.~Kusner and N.~Schmitt,
The spinor representation of surfaces in space,
arXiv:dg-ga/9610005 (1996).

\bibitem{LM1989}
H.~B.~Lawson and M.-L.~Michelsohn,
\emph{Spin Geometry},
Princeton University Press, Princeton, 1989.

\bibitem{Li1995JDE}
Y.~Y.~Li,
Prescribing scalar curvature on $S^n$ and related problems,
Part I,
\emph{J. Differential Equations} \textbf{120} (1995), 319--410.
\href{https://doi.org/10.1006/jdeq.1995.1115}
{doi:10.1006/jdeq.1995.1115}.

\bibitem{LiZhu1999}
Y.~Y.~Li and M.~Zhu,
Yamabe type equations on three dimensional Riemannian manifolds,
\emph{Commun. Contemp. Math.} \textbf{1} (1999),
no.~1, 1--50.
doi:10.1142/S021919979900002X.

\bibitem{Matsutani1999}
S.~Matsutani,
Immersion anomaly of Dirac operator on surface in $\mathbb R^3$,
\emph{Rev. Math. Phys.} \textbf{11} (1999),
no.~2, 171--186.
doi:10.1142/S0129055X99000076.

\bibitem{Schoen1989}
R.~M.~Schoen,
Variational theory for the total scalar curvature functional
for Riemannian metrics and related topics,
in \emph{Topics in Calculus of Variations},
Lecture Notes in Mathematics, vol.~1365,
Springer, Berlin, 1989, pp.~120--154.
doi:10.1007/BFb0089180.

\bibitem{Taimanov1998}
I.~A.~Taimanov,
The Weierstrass representation of closed surfaces
in $\mathbb R^3$,
\emph{Funct. Anal. Appl.} \textbf{32} (1998),
no.~4, 258--267.
doi:10.1007/BF02463208.

\bibitem{Taimanov1999}
I.~A.~Taimanov,
The Weierstrass representation of spheres in $\mathbb R^3$,
the Willmore numbers, and soliton spheres,
\emph{Proc. Steklov Inst. Math.} \textbf{225} (1999),
322--343.

\bibitem{WangZhang2025}
G.~Wang and M.~Zhang,
Stability of spinorial Sobolev inequalities on $\mathbb S^n$,
\emph{Adv. Math.} \textbf{505} (2026), 111318.

\bibitem{Xu2024}
T.~Xu,
Conformal embeddings of $S^2\to\mathbb R^3$ with prescribed
mean curvature: a variational approach,
\emph{Commun. Anal. Geom.} \textbf{32} (2024), 203--254.
\href{https://doi.org/10.4310/cag.240905042216}
{doi:10.4310/cag.240905042216}.

\end{thebibliography}
\end{document}